\documentclass{amsart}

\usepackage{amssymb,amsfonts}
\usepackage[all,arc]{xy}
\usepackage{enumerate}
\usepackage{mathrsfs}
\usepackage{url}
\usepackage{mathtools}

\newtheorem{thm}{Theorem}[section]
\newtheorem{cor}[thm]{Corollary}
\newtheorem{prop}[thm]{Proposition}
\newtheorem{lem}[thm]{Lemma}

\theoremstyle{definition}
\newtheorem{defn}[thm]{Definition}

\newtheorem{fact}[thm]{Fact}
\newtheorem{obs}[thm]{Observation}

\theoremstyle{remark}
\newtheorem{rem}[thm]{Remark}

\global\long\def\tp{\operatorname{tp}}

\global\long\def\cdt{\operatorname{cdt}}

\global\long\def\inp{\operatorname{inp}}

\global\long\def\sct{\operatorname{sct}}

\global\long\def\qftp{\operatorname{qftp}}

\global\long\def\alg{\operatorname{alg}}

\global\long\def\pfc{\operatorname{pfc}}

\global\long\def\acl{\operatorname{acl}}

\global\long\def\SOP{\operatorname{SOP}}

\global\long\def\ACF{\operatorname{ACF}}

\global\long\def\NSOP{\operatorname{NSOP}}

\global\long\def\Sk{\operatorname{Sk}}

\global\long\def\TP{\operatorname{TP}}

\newcommand{\M}{\mathbb{M}}

\makeatletter
\let\c@equation\c@thm
\makeatother
\numberwithin{equation}{section}

\def\Ind{\setbox0=\hbox{$x$}\kern\wd0\hbox to 0pt{\hss$\mid$\hss} \lower.9\ht0\hbox to 0pt{\hss$\smile$\hss}\kern\wd0} 

\def\Notind{\setbox0=\hbox{$x$}\kern\wd0\hbox to 0pt{\mathchardef \nn=12854\hss$\nn$\kern1.4\wd0\hss}\hbox to 0pt{\hss$\mid$\hss}\lower.9\ht0 \hbox to 0pt{\hss$\smile$\hss}\kern\wd0} 

\def\ind{\mathop{\mathpalette\Ind{}}} 

\def\nind{\mathop{\mathpalette\Notind{}}} 

\usepackage{cite}

\title{On Model-Theoretic Tree Properties}

\author{Artem Chernikov and Nicholas Ramsey}\thanks{A.C. was partially supported by ValCoMo (ANR-13-BS01-0006), by EPSRC grant EP/K020692/1, by the Fondation Sciences Mathematiques de Paris
(FSMP) and by a public grant overseen by the French National Research Agency (ANR)
as part of the Investissements d'avenir program (reference: ANR-10-LABX-0098)}

\date{\today}

\begin{document}

\begin{abstract}
We study model theoretic tree properties ($\TP, \TP_1, \TP_2$) and their associated cardinal invariants ($\kappa_{\cdt}, \kappa_{\sct}, \kappa_{\inp}$, respectively). In particular, we obtain a quantitative refinement of Shelah's theorem ($\TP \Rightarrow \TP_1 \lor \TP_2$) for countable theories,  show that $\TP_1$ is always witnessed by a formula in a single variable (partially answering a question of Shelah) and that weak $k-\TP_1$ is equivalent to $\TP_1$ (answering a question of Kim and Kim). Besides, we give a characterization of $\NSOP_1$ via a version of independent amalgamation of types and apply this criterion to verify that some examples in the literature are indeed $\NSOP_1$.
\end{abstract}

\maketitle

\setcounter{tocdepth}{1}
\tableofcontents

\section{Introduction}

One of the central tasks of abstract model theory is to understand what kinds of complete first-order theories there are and how complicated they can be.  In practice, this is achieved by classifying theories according to the combinatorial configurations that do or do not appear among the definable sets in their models.  The most meaningful of these configurations, the so-called \emph{dividing lines}, have the property that their absence signals the existence of some positive structure, while their presence indicates some kind of complexity.  Dividing lines come in two flavors:  local properties, which describe the combinatorics of sets defined by instances of a single formula, and global properties, which describe the interaction of definable sets generally.  Stability, simplicity, NIP are examples of the former, while \(\omega\)-stability, supersimplicity, and strong dependence are examples of the latter (see e.g. \cite{Conant}).  

In this paper, we study some questions around Shelah's tree property $\TP$ and its relatives $\SOP_1$, $\TP_1$, $\TP_2$ and weak \(k\)-TP\(_{1}\), as well as their global analogues detected by the cardinal invariants \(\kappa_{\text{cdt}}(T)\), \(\kappa_{\text{inp}}(T)\), and \(\kappa_{\text{sct}}(T)\).  Our point of departure is the third chapter of Shelah's \emph{Classification Theory}.  There, Shelah investigates the global combinatorics of stable theories in terms of a cardinal invariant \(\kappa(T)\) quantifying the complexity of forking in models of \(T\).  In the final section of this chapter, he introduces variations on \(\kappa(T)\) with the invariants \(\kappa_{\text{cdt}}(T)\), \(\kappa_{\text{sct}}(T)\), and \(\kappa_{\text{inp}}(T)\) and proves several results about how they relate.  In contemporary language, these invariants bound the size of approximations to the tree property, the tree property of first kind, and the tree property of the second kind consistent with \(T\), respectively.  Later as the theory developed, a property of stable theories that forking satisfies \emph{local character} was isolated and theories satisfying this condition, the \emph{simple theories}, were intensively studied \cite{shelah1980simple, wagner2000simple, casanovas2011simple}.  These theories are exactly the theories without the tree property, which is to say those theories with \(\kappa_{\text{cdt}}(T)\) bounded.  Nonetheless, until recently, the aforementioned invariants have received very little attention and many basic questions remain unaddressed.  

Here, we focus on two such questions.  Shelah proved that a theory has the tree property if and only if it has the tree property of the first kind or the tree property of the second kind \cite{ShelahCT}.  In terms of the invariants, this amounts to the assertion that \(\kappa_{\text{cdt}}(T) = \infty\) if and only if \(\kappa_{\text{inp}}(T) + \kappa_{\text{sct}}(T) = \infty\).  It is natural to ask if this relationship persists when \(\kappa_{\text{cdt}}(T)\) is bounded --- in other words, if the equality \(\kappa_{\text{cdt}}(T) = \kappa_{\text{inp}}(T) + \kappa_{\text{sct}}(T)\) holds in general. Shelah also proved that \(\kappa_{\text{cdt}}(T) = \kappa\) is always witnessed by a sequence of formulas in a single free variable when \(\kappa\) is an infinite cardinal or \(\infty\).  Recently, the first named author proved an analogous result for \(\kappa_{\inp}(T)\) \cite{MR3129735}.  We consider here whether or not the computation of \(\kappa_{\text{sct}}(T)\) similarly reduces to a single free variable.  These questions were both raised by Shelah (Question 7.14 in \cite{ShelahCT}).   

We do not give a complete answer to any of them, but for each of these questions there are two model-theoretically natural special cases to consider:  first, the case of countable theories and, secondly, the case where one or more of the invariants in question are unbounded (which reduces to a question about configurations in a single formula).   In Section \ref{Invariants}, we show that \(\kappa_{\text{cdt}}(T) = \kappa_{\text{inp}}(T) + \kappa_{\text{sct}}(T)\) for countable \(T\).  In Section \ref{TP1}, we show that if \(\kappa_{\text{sct}}(T) = \infty\) then this will be witnessed by a formula in a single free variable by showing that $\TP_1$ is always witnessed by a formula in one free variable.  The main ingredient in our argument is the notion of a \emph{strongly indiscernible tree}, which is more easily manipulated than the \(s\)-indiscernible trees used in other studies of the tree property of the first kind.  

At the present state of the theory, the class of non-simple theories without the strict order property is poorly understood even at the level of syntax.  In their study of the order \(\unlhd^{*}\), Dzamonja and Shelah introduced a weakening of $\TP_1$ called $\SOP_1$ \cite{DjSh:692}.  Subsequently, Kim and Kim introduced two infinite families of properties called \(k\)-TP\(_{1}\) and \emph{weak} \(k\)-TP$_1$ for \(k \geq 2\) and showed 
$$
\TP_1 \iff k\text{-TP}_{1} \iff \text{weak }2\text{-TP}_{1} \implies \text{weak }3\text{-TP}_{1} \implies \ldots \implies \SOP_1
$$
It was left open whether the properties weak \(k\)-TP\(_{1}\) are inequivalent for distinct \(k\) and whether or not weak \(k\)-TP\(_{1}\) is equivalent to $\TP_1$ \cite{KimKimNTP1}.  In our work on proving that $\TP_1$ is witnessed by a formula in one free variable, we obtained unexpectedly a simple and direct proof that the weak \(k\)-TP$_1$ hierarchy collapses and that they are all equivalent to $\TP_1$.  

In the final two sections of the paper, we study theories without the property $\SOP_1$.  We show that independent amalgamation fails in a strong way in theories with $\SOP_1$ and that they are in fact characterized by this feature.  This gives rise to a useful criterion for showing that a theory is NSOP\(_{1}\) (and hence NTP\(_{1}\)).  Leveraging work of Granger \cite{Granger} and Chatzidakis \cite{ZoeFree}, this allows us to conclude that both the two sorted theory of infinite-dimensional vector spaces over an algebraically closed field with a generic bilinear form, as well as the theory of \(\omega\)-free PAC fields of characteristic zero are NSOP\(_{1}\).  Finally, we generalize the construction of the theory of parametrized equivalence relations \(T^{*}_{\text{feq}}\) to give a general method for constructing NSOP\(_{1}\) theories from simple ones.  We learned after this work was completed that essentially the same construction had been studied by Baudisch \cite{Baudisch}, but our emphasis is different.  We show that the independence theorem holds for these structures, allowing us to obtain a proof that \(T^{*}_{\text{feq}}\) is NSOP\(_{1}\) as a corollary.  

\subsection*{Acknowledgements}

We would like to thank the referee for numerous suggestions on improving the presentation, Zo\'e Chatzidakis for her help with Lemma \ref{lem: strong finite char PAC}, and Alex Kruckman for pointing out an error and a way to fix it in Section \ref{sec: parametrization} of an earlier version of the article.

\section{Preliminaries on indiscernible trees}

We fix a complete first-order theory $T$ in a language $L$, $\M \models T$ is a monster model. In several of the arguments below, we will make use of the notion of an indiscernible tree.  For our purposes, there are two different languages we will need to place on the index model:  \(L_{s,\lambda} = \{\vartriangleleft, \wedge, <_{lex}, (P_{\alpha} : \alpha < \lambda)\}\) and \(L_{0} = \{ \vartriangleleft, \wedge, <_{lex}\}\) where \(\lambda\) is a cardinal.  We may view the tree \(\kappa^{<\lambda}\) as an \(L_{s,\lambda}-\) or \(L_{0}\)-structure in a natural way, interpreting \(\vartriangleleft\) as the tree partial order, \(\wedge\) as the binary meet function, \(<_{lex}\) as the lexicographic order, and \(P_{\alpha}\) as a predicate which identifies the \(\alpha\)th level (we will only consider \(\kappa = 2\) and \(\kappa = \omega\)).  See \cite{KimKimScow} and \cite{TakeuchiTsuboi} for more details.  

\begin{defn}
Suppose that \((a_{\eta})_{\eta \in \kappa^{<\lambda}}\) and \((a_{\alpha,i})_{\alpha < \kappa, i < \omega}\) are collections of tuples and \(C\) is a set of parameters in some model.
\begin{enumerate}
\item We say \((a_{\eta})_{\eta \in \kappa^{<\lambda}}\) is an \(s\)\emph{-indiscernible tree over $C$} if
\[
\text{qftp}_{L_{s,\lambda}}(\eta_{0}, \ldots, \eta_{n-1}) = \text{qftp}_{L_{s,\lambda}}(\nu_{0}, \ldots, \nu_{n-1})
\] 
implies \(\text{tp}(a_{\eta_{0}}, \ldots, a_{\eta_{n-1}}/C) = \text{tp}(a_{\nu_{0}}, \ldots, a_{\nu_{n-1}}/C)\), for all $n \in \omega$.  
\item We say \((a_{\eta})_{\eta \in \kappa^{<\lambda}}\) is a \emph{strongly indiscernible tree over $C$} if 
\[
\text{qftp}_{L_{0}}(\eta_{0}, \ldots, \eta_{n-1}) = \text{qftp}_{L_{0}}(\nu_{0}, \ldots, \nu_{n-1})
\]
implies \(\text{tp}(a_{\eta_{0}}, \ldots, a_{\eta_{n-1}}/C) = \text{tp}(a_{\nu_{0}}, \ldots, a_{\nu_{n-1}}/C)\), for all $n \in \omega$.
\item We say \((a_{\alpha,i})_{\alpha < \kappa, i < \lambda}\) is a \emph{mutually indiscernible array over $C$} if, for all $\alpha < \kappa$, $(a_{\alpha, i})_{i < \lambda}$ is a sequence indiscernible over $C \cup\{a_{\beta,j} : \beta < \kappa, \beta \neq \alpha, j < \lambda\}$.
\end{enumerate}
\end{defn}

\begin{lem} \label{StronglyIndiscTreeProp}
Let $(a_\eta : \eta \in \kappa^{<\lambda})$ be a tree strongly indiscernible over a set of parameters $C$.
\begin{enumerate}
\item All paths have the same type over $C$: for any $\eta, \nu \in \kappa^{\lambda}$, $\tp((a_{\eta | \alpha} : \alpha < \lambda)/C) = \tp((a_{\nu|\alpha} : \alpha < \lambda)/C)$.
\item For any $\eta \perp \nu \in \kappa^{< \lambda}$ and any $\xi$, $\tp(a_\eta, a_\nu/C) = \tp(a_{\xi \frown 0}, a_{\xi \frown 1}/C)$.
\item The tree $(a_{0 \frown \eta} : \eta \in \kappa^{< \lambda})$ is strongly indiscernible over $a_{\emptyset} C$.
\end{enumerate}
\end{lem}

\begin{proof}
(1) This follows by strong indiscernibility of the tree as for any $\eta, \nu \in \kappa^{<\lambda}$, $\text{qftp}_{L_0}((\eta|\alpha : \alpha < \lambda)) = \qftp_{L_0}((\nu|\alpha : \alpha < \lambda))$.

(2) Let $\eta \perp \nu \in \kappa^{<\lambda}$ be given, without loss of generality \(\eta <_{lex} \nu\) and let \(\mu = \eta \wedge \nu\).  Then there are \(i < j < \kappa \) so that \(\mu \frown \langle i\rangle \unlhd \eta\) and \(\mu \frown \langle j \rangle \unlhd \nu\).  Then $\qftp_{L_0}(\eta, \nu) = \qftp_{L_0}(\mu \frown \langle i \rangle, \mu \frown \langle j \rangle) = \qftp_{L_0}(\mu \frown 0, \mu \frown 1) = \qftp_{L_0}(\xi \frown 0, \xi \frown 1)$, and we conclude by strong indiscernibility of the tree.

(3) Clear as $\qftp_{L_0}(\bar{\eta}) = \qftp_{L_0}(\bar{\nu})$ implies $\qftp_{L_0} (\bar{\eta}, \emptyset) = \qftp_{L_0} (\bar{\nu}, \emptyset)$, provided $\emptyset$ is not enumerated in neither $\overline{\eta}$ nor $\overline{\nu}$.  
\end{proof}

\begin{lem}\label{s-IndiscTreeProp} 
Let $(a_\eta : \eta \in \kappa^{<\lambda})$ be a tree s-indiscernible over a set of parameters $C$.
\begin{enumerate}
\item All paths have the same type over $C$: for any $\alpha, \nu \in \kappa^{\lambda}$, $\tp((a_{\eta | \alpha})_{\alpha < \lambda}/C) = \tp((a_{\nu|\alpha})_{\alpha < \lambda}/C)$.
\item Suppose $\{\eta_{\alpha} : \alpha < \gamma\} \subseteq \kappa^{<\lambda}$ satisfies $\eta_{\alpha} \perp \eta_{\alpha'}$ whenever $\alpha \neq \alpha'$.  Then the array $(b_{\alpha, \beta})_{\alpha < \gamma, \beta < \kappa}$ defined by 
$$
b_{\alpha, \beta} = a_{\eta_{\alpha} \frown \langle \beta \rangle}
$$
is mutually indiscernible over $C$.  
\end{enumerate}
\end{lem}

\begin{proof}
(1) This follows by $s$-indiscernibility of the tree as for any $\eta, \nu \in \kappa^{<\lambda}$, $\text{qftp}_{L_s}((\eta|\alpha : \alpha < \lambda)) = \qftp_{L_s}((\nu|\alpha : \alpha < \lambda))$.

(2) Fix $\alpha < \gamma$ and let $A = \{a_{\eta_{\alpha' \frown \langle \beta \rangle}} : \alpha \neq \alpha' < \gamma, \beta < \kappa\}\cup C$.  As the elements of $\{\eta_{\alpha} : \alpha < \gamma\}$ are pairwise incomparable, it is easy to check that for any $\beta_{0} < \ldots < \beta_{n-1} < \kappa$ and $\beta'_{0} < \ldots < \beta'_{n-1} < \kappa$, 
$$
\text{qftp}_{L_{s}}(a_{\eta_{\alpha} \frown \langle \beta_{0} \rangle}, \ldots, a_{\eta_{\alpha} \frown \langle \beta_{n-1} \rangle }/A) = \text{qftp}_{L_{s}}(a_{\eta_{\alpha} \frown \langle \beta'_{0} \rangle}, \ldots, a_{\eta_{\alpha} \frown \langle \beta'_{n-1} \rangle }/A), 
$$
which proves (2).
\end{proof}

Now we note that \(s\)-indiscernible and strongly indiscernible trees exist.  

\begin{defn}
Suppose \(I\) is an \(L'\)-structure, where $L'$ is some language. We say that \(I\)-indexed indiscernibles have the \emph{modeling property} if, given any \((a_{i} : i \in I)\) from $\M$, there is an \(I\)-indexed indiscernible \((b_{i} : i \in I)\) in $\M$ \emph{locally based} on the \((a_{i})\):  given any finite set of formulas \(\Delta\) from \(L\) and a finite tuple \((t_{0}, \ldots, t_{n-1})\) from \(I\), there is a tuple \((s_{0}, \ldots, s_{n-1})\) from \(I\) so that 
\[
\qftp_{L'} (t_{0}, \ldots, t_{n-1}) =\qftp_{L'}(s_{0}, \ldots , s_{n-1})
\]
and also 
\[
\text{tp}_{\Delta}(b_{t_{0}}, \ldots, b_{t_{n-1}}) = \text{tp}_{\Delta}(a_{s_{0}}, \ldots, a_{s_{n-1}}).
\]
\end{defn}

\begin{fact}\cite{TakeuchiTsuboi, Scow, KimKimScow}\label{modeling}
Let \(I_{0}\) denote the \(L_{0}\)-structure \((\omega^{<\omega}, \unlhd, <_{lex}, \wedge)\) and \(I_{s}\) be the \(L_{s,\omega}\)-structure \((\omega^{<\omega}, \unlhd, <_{lex}, \wedge, (P_{\alpha})_{\alpha < \omega})\) with all symbols being given their intended interpretations and each \(P_{\alpha}\) naming the elements of the tree at level \(\alpha\).  Then strongly indiscernible trees (\(I_{0}\)-indexed indiscernibles) and \(s\)-indiscernible trees (\(I_{s}\)-indexed indiscernibles) have the \emph{modeling property}.  
\end{fact}

In the arguments below, we will often argue by induction where at each stage it is necessary to modify a tree of tuples in a way that maintains the indiscernibility of the tree.  A convenient way of organizing these arguments is to make a catalogue of operations on indiscernible trees and prove that these operations preserve the relevant indiscernibility.  

\begin{defn}
Fix \(k \geq 1\).  
 \begin{enumerate}
  \item (widening) The \emph{k-fold widening of} \((a_{\eta})_{\eta \in \omega^{<\omega}}\) \emph{at level} \(n\) is defined to be the tree \((a'_{\eta})_{\eta \in \omega^{<\omega}}\) where 
\[
 a'_{\eta} = \left\{
\begin{matrix}
 a_{\eta} & \text{ if } l(\eta) < n \\
(a_{\nu \frown (ki) \frown \xi}, \ldots, a_{\nu \frown (ki + (k-1))\frown \xi}) & \text{ if } \eta = \nu \frown i \frown \xi \\ & \text{ where } \nu \in \omega^{n-1}, i \in \omega, \xi \in \omega^{<\omega}. \\
\end{matrix}
\right.
\]
\item (stretching) The \emph{k-fold stretch of }\((a_{\eta})_{\eta \in \omega^{<\omega}}\)\emph{ at level }\(n\) is defined to be the tree \((a''_{\eta})_{\eta \in \omega^{<\omega}}\) where 
\[
 a''_{\eta} = \left\{ 
\begin{matrix}
 a_{\eta} & \text{ if } l(\eta) < n \\
(a_{\eta}, a_{\eta \frown 0}, \ldots, a_{\eta \frown 0^{k-1}}) & \text{ if } l(\eta) = n \\
a_{\nu \frown 0^{k-1} \frown \xi} & \text{ if } \eta = \nu \frown \xi \text{ for } \nu \in \omega^{n}, \xi \neq \emptyset
\end{matrix}
\right.
\]
\item (fattening) Given a tree \((a_{\eta})_{\eta \in 2^{<\kappa}}\), define the \(k\)\emph{-fold fattening of }\((a_{\eta})_{\eta \in 2^{<\kappa}}\) to be the tree \((a^{(k)}_{\eta})_{\eta \in 2^{<\kappa}}\) by induction as follows: for each \(\eta \in 2^{<\kappa}\) let \(a^{(0)}_{\eta} = a_{\eta}\).  If \((a^{(n)}_{\eta})_{\eta \in 2^{<\kappa}}\) has been defined, for each \(\eta \in 2^{<\kappa}\), let \(a^{(n+1)}_{\eta} = (a^{(n)}_{0 \frown \eta}, a^{(n)}_{1 \frown \eta})\).  Let \(C_{k} = \{a_{\eta} : \eta \in 2^{< k}\}\), the \emph{stump below} \(k\).  Set \(C_{0} = \emptyset\).  
\item (restricting)  Given the tree \((a_{\eta})_{\eta \in \lambda^{<\kappa}}\) and \(W \subseteq \kappa\), we define the \emph{restriction of }\((a_{\eta})_{\eta \in \lambda^{<\kappa}}\)\emph{ to }\(W\) to be the collection of tuples 
\[
\{a_{\eta} : l(\eta) \in W \text{ and if } \beta \not\in W, \text{ then } \eta(\beta) = 0\}.
\] 
If the order type of \(W\) is \(\alpha\), the restriction of \((a_{\eta})_{\eta \in \lambda^{<\kappa}}\) may be naturally identified with \((a_{\eta})_{\eta \in \lambda^{<\alpha}}\).  
\item (elongating)  Given \(\eta \in \kappa^{<\omega}\), with \(l(\eta) = n\), define \(\tilde{\eta} \in \kappa^{<\omega}\) to be the tuple with length \(k(l(\eta) - 1) + 1\) defined by 
\[
\tilde{\eta}(i) = \left\{
\begin{matrix}
\eta(i/k) & \text{ if } k|i \\
0 & \text{ otherwise}
\end{matrix}
\right.
\]
Then define the \(k\)\emph{-fold elongation} of \((a_{\eta})_{\eta \in \kappa^{<\omega}}\) to be the tree \((b_{\eta})_{\eta \in \kappa^{<\omega}}\) where 
\[
b_{\eta} = (a_{\tilde{\eta}}, a_{\tilde{\eta} \frown 0}, \ldots, a_{\tilde{\eta} \frown 0^{k-1}}).
\]
 \end{enumerate}
\end{defn}
\vspace{.25in}

\includegraphics[trim = 25mm 0mm 0mm 0mm, clip, scale=.43]{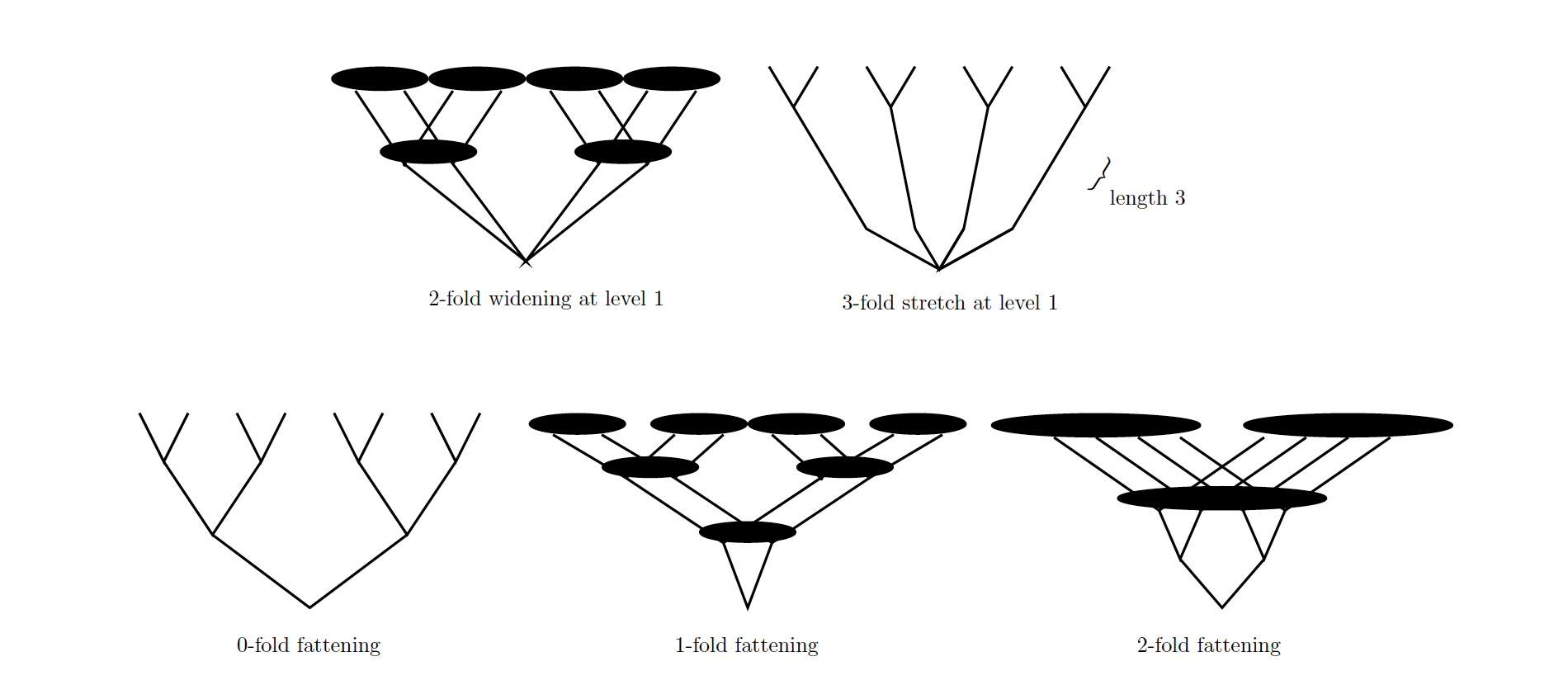} \hspace{2in}

\begin{prop}\label{preservation}
\begin{enumerate}
\item \(s\)-indiscernibility is preserved under widening, stretching, fattening, restriction, and elongating.  
\item Strong indiscernibility is preserved under restriction, fattening, and elongating.  Moreover, if \((a_{\eta})_{\eta \in 2^{<\omega}}\) is strongly indiscernible, then the \(k\)-fold fattening \((a^{(k)})_{\eta \in 2^{<\omega}}\) is strongly indiscernible over \(C_{k}\).  
\end{enumerate}
\end{prop}

\begin{proof}
The proofs of these facts can be found in Section \ref{sec: appendix}.  
\end{proof}

\section{Cardinal invariants and tree properties}\label{Invariants}

\begin{defn}
Suppose $T$ is a complete theory and $\varphi(x;y) \in L$ is a formula in the language of $T$.
\begin{enumerate}
\item $\varphi(x;y)$ has the \emph{tree property} ($\TP$) if there is $k < \omega$ and a tree of tuples $(a_{\eta})_{\eta \in \omega^{<\omega}}$ in $\M$ such that
\begin{itemize}
\item for all $\eta \in \omega^{\omega}$, $\{\varphi(x;a_{\eta | \alpha}) : \alpha < \omega\}$ is consistent,
\item for all $\eta \in \omega^{<\omega}$, $\{\varphi(x;a_{\eta \frown \langle i \rangle}) : i < \omega\}$ is $k$-inconsistent.  
\end{itemize}
\item $\varphi(x;y)$ has the \emph{tree property of the first kind} ($\TP_1$) if there is a tree of tuples $(a_{\eta})_{\eta \in \omega^{<\omega}}$ in $\M$ such that
\begin{itemize}
\item for all $\eta \in \omega^{\omega}$, $\{\varphi(x;a_{\eta | \alpha}) : \alpha < \omega\}$ is consistent,
\item for all $\eta \perp \nu$ in $\omega^{<\omega}$, $\{\varphi(x;a_{\eta}),\varphi(x;a_{\nu})\}$ is inconsistent. 
\end{itemize}
\item $\varphi(x;y)$ has the \emph{tree property of the second kind} ($\TP_{2}$) if there is a $k < \omega$ and an array $(a_{\alpha,i})_{\alpha < \omega, i < \omega}$ in $\M$ such that 
\begin{itemize}
\item for all functions $f: \omega \to \omega$, $\{\varphi(x;a_{\alpha, f(\alpha)}) : \alpha < \omega\}$ is consistent,
\item for all $\alpha$, $\{\varphi(x;a_{\alpha, i}) : i < \omega\}$ is $k$-inconsistent.  
\end{itemize}
\item $T$ has one of the above properties if some formula does modulo $T$.  
\end{enumerate}
\end{defn}

It is easy to see that if a theory has the tree property of the first or second kind, then it also has the tree property.  Remarkably, the converse is also true.

\begin{fact}\cite{ShelahCT} \label{dichotomy}
A complete theory $T$ has $\TP$ if and only if it has $\TP_1$ or $\TP_2$. 
\end{fact}

The above theorem was first proven in different language, before any of the three properties were actually defined.  The purpose of this section is to prove a refinement of this theorem, by studying the relationship between approximations to the tree property and those to the tree property of the first or second kind.  In order to do so, however, it will be necessary to return to the vocabulary in which Fact \ref{dichotomy} was initially formulated.  

\begin{defn}
The following notions were introduced in \cite{ShelahCT}.
\begin{enumerate}
\item A \emph{$\cdt$-pattern of depth} \(\kappa\) is a sequence of formulas \(\varphi_{i}(x;y_{i})\) (\(i < \kappa, i \text{ successor}\)) and numbers \(n_{i} < \omega\), and a tree of tuples \((a_{\eta})_{\eta \in \omega^{<\kappa}}\) for which 
\begin{enumerate}
\item \(p_{\eta} = \{\varphi_{i}(x;a_{\eta | i}) : i \text{ successor }, i < \kappa\}\) is consistent for \(\eta \in \omega^{\kappa}\), 
\item \(\{\varphi_{i} (x;a_{\eta \frown \langle \alpha \rangle}) : \alpha < \omega , i = l(\eta) + 1\}\) is \(n_{i}\)-inconsistent.
\end{enumerate}
A cdt-pattern with $n_{i} \leq n$ for all $i < \kappa$, is called a $($cdt$,n)$-pattern.  
\item An \emph{$\inp$-pattern of depth} \(\kappa\) is a sequence of formulas \(\varphi_{i}(x;y_{i})\) \((i < \kappa)\), sequences \((a_{i,\alpha}: \alpha < \omega)\), and numbers \(n_{i} <\omega\) such that 
\begin{enumerate}
\item for any \(\eta \in \omega^{\kappa}\), \(\{ \varphi_{i}(x;a_{i,\eta(i)}) : i < \kappa\}\) is consistent,
\item for any \(i < \kappa\), \(\{\varphi_{i}(x;a_{i,\alpha}) : \alpha < \omega\}\) is \(n_{i}\)-inconsistent.  
\end{enumerate}
\item An \emph{$\sct$-pattern of depth} \(\kappa\) is a sequence of formulas \(\varphi_{i}(x;y_{i})\) \((i < \kappa)\) and a tree of tuples \((a_{\eta})_{\eta \in \omega^{<\kappa}}\) such that 
\begin{enumerate}
\item for every \(\eta \in \omega^{\kappa}\), \(\{\varphi_{\alpha}(x;a_{\eta | \alpha}) : 0 < \alpha < \kappa, \alpha \text{ successor}\}\) is consistent,
\item If \(\eta \in \omega^{\alpha}\), \(\nu \in \omega^{\beta}\), \(\alpha, \beta\) are successors, and \(\nu \perp \eta\) then the formulas $\{\varphi_{\alpha}(x;a_{\eta}), \varphi_{\beta}(x;a_{\nu})\}$ are inconsistent.  
\end{enumerate}
If instead of (b), we have:  for any pairwise incomparable $(\eta_{i} : i < k)$, $\{\varphi_{l(\eta_{i})}(x;a_{\eta_{i}}) : i < k\}$ is inconsistent, then we call this a $(\text{sct},k)$-pattern.  
\item For \(X \in \{\text{cdt}, \text{sct}, \text{inp}\}\), we define \(\kappa_{X}^{n}(T)\) to be the first cardinal \(\kappa\) so that there is {\bf no} \(X\)-pattern of depth \(\kappa\) in \(n\) free variables, and $\infty$ if no such $\kappa$ exists.  We define \(\kappa_{X}(T) = \sup_{n \in \omega} \{\kappa_{X}^{n}\}\).
\end{enumerate}
\end{defn}

\begin{rem}
We note that the notion of a $(\text{cdt},n)$-pattern \emph{strengthens} that of a cdt-pattern by imposing a uniform finite bound on the size of the inconsistency at each level, while the notion of an $(\text{sct},n)$-pattern \emph{weakens} that of an sct-pattern by only requiring any $n$ incomparable elements to be inconsistent rather than any $2$.  One can regard an $(\text{sct},n)$-pattern as an approximation to a witness to $n$-TP$_1$ (see Definition \ref{weak} below).  
\end{rem}

\begin{obs}\label{basicobs}
Fix a complete theory $T$. 
\begin{enumerate}
\item $\kappa^{n}_{\text{sct}}(T) \geq n$, $\kappa^{n}_{\text{inp}}(T) \geq n$ and $\kappa^{n}_{\text{cdt}}(T) \geq n$ for all $n$.
\item 
\begin{enumerate}
\item $\kappa_{\text{cdt}}(T) = \infty$ if and only if $\kappa_{\text{cdt}}(T) > |T|^{+}$ if and only if $T$ has $\TP$.
\item $\kappa_{\text{sct}}(T) = \infty$ if and only if $\kappa_{\text{sct}}(T) > |T|^{+}$ and only if $T$ has $\TP_1$.
\item $\kappa_{\text{inp}}(T) = \infty$ if and only if $\kappa_{\text{inp}}(T) > |T|^{+}$ if and only if $T$ has $\TP_2$.  
\end{enumerate}
\item $\max\{ \kappa_{\text{sct}}^{n}(T), \kappa_{\text{inp}}^{n}(T)\} \leq \kappa^{n}_{\text{cdt}}(T)$.  
\end{enumerate}
\end{obs}

\begin{proof}
(1) follows from the fact that ``$=$" is in the language.  

(2)  As each case is entirely similar, we'll sketch the argument for (a) only.  If $\kappa_{\text{cdt}}\left(T\right)>|T|^{+}$, then in the pattern witnessing it we may assume that $\varphi_{i}\left(x,y_{i}\right)=\varphi\left(x,y\right)$ and $k_{i}=k$, because $\left|T\right|\geq \aleph_{0}$.  This is a witness to $\TP$. And then using compactness we can find a pattern witnessing that $\kappa_{\text{cdt}}^{n}(T)>\kappa$ for any cardinal $\kappa$.  

(3)  If \(\varphi_{i}(x;y_{i})\) \((i < \kappa)\), \((a_{i,\alpha}: \alpha < \omega)\), \((n_{i})_{i <\omega}\) form an inp-pattern of depth $\kappa$, obtain a cdt-pattern of depth $\kappa$ with respect to the same formulas by defining $(b_{\eta})_{\eta \in \omega^{<\kappa}}$ by $b_{\eta} = a_{l(\eta), \eta(l(\eta)-1)}$.
\end{proof}

\begin{lem}\label{s-witness}
(1)  If there is an sct-pattern (cdt-pattern) of depth $\kappa$ modulo $T$, then there is an sct-pattern (cdt-pattern) $\varphi_{\alpha}(x;y_{\alpha})$, $(a_{\eta})_{\eta \in \omega^{<\kappa}}$ in the same number of free variables so that $(a_{\eta})_{\eta \in \omega^{<\kappa}}$ is an $s$-indiscernible tree. 

(2)  If there is an inp-pattern of depth $\kappa$ modulo $T$, then there is an inp-pattern $\varphi_{\alpha}(x;y_{\alpha})$ $(\alpha < \kappa)$, $(k_{\alpha})_{\alpha < \kappa}$, $(a_{\alpha, i})_{\alpha < \kappa, i < \omega}$ in the same number of free variables so that $(a_{\alpha, i})_{\alpha < \kappa, i < \omega}$ is a mutually indiscernible array.  
\end{lem}

\begin{proof}
(1)  By compactness and Fact \ref{modeling}.  

(2) This is Lemma 2.2 of \cite{MR3129735}.  
\end{proof}

Now we fix a complete theory \(T\) and for \(X \in \{\text{cdt}, \text{sct}, \text{inp}\}\), we write \(\kappa_{X}\) for \(\kappa_{X}(T)\).

\begin{prop} \label{prop: finding sct,k pattern}
Assume that $\kappa_{\cdt}^{n}\geq\aleph_{0}$. Then either $\kappa_{\inp}^{n}\geq\aleph_{0}$
or $\kappa_{\sct,k}^{n}\geq\aleph_{0}$ for some $k\in\omega$ (i.e. there are $(\kappa_{\sct}, k)$-patterns in $n$ variables of arbitrary finite depth).  In fact, if $\kappa^{n}_{\inp} < \aleph_{0}$, then one can take $k = \kappa^{n}_{\inp}$.  
\end{prop}
\begin{proof}
If $\kappa_{\inp}^{n}\geq\aleph_{0}$ does not hold, then in fact
we have $\kappa_{\inp}^{n}\leq k$ for some $k\in\omega$. 

Fix an arbitrary $m\in\omega$, then by assumption and Lemma \ref{s-witness}
we can find $\left(a_{\eta}:\eta\in\omega^{<2m}\right),\left(\varphi_{i}\left(x,y_{i}\right):i<2m\right),\left(k_{i}:i<2m\right)$
an $s$-indiscernible $\cdt$-pattern with $\left|x\right|=n$, i.e.:
\begin{enumerate}
\item $\left(a_{\eta}:\eta\in\omega^{<2m}\right)$ is an $s$-indiscernible
tree,
\item $\left\{ \varphi_{i}\left(x,a_{\eta\restriction i}\right):i<2m\right\} $
is consistent for every $\eta\in\omega^{2m}$,
\item $\left\{ \varphi_{i}\left(x,a_{\eta \frown \langle j\rangle}\right):j\in\omega\right\} $ is
$k_{i}$-inconsistent for every $i<2m-1$ and $\eta\in\omega^{i}$.
\end{enumerate}
For $l<m$ and $\nu\in\omega^{l}$ we define $\nu^{*}=\left(\nu\left(0\right),0,\nu\left(1\right),0,\ldots,\nu\left(l-1\right),0\right)\in\omega^{<2m}$.
Let $\left\{ \nu_{0},\ldots,\nu_{k-1}\right\} \subseteq\omega^{<m}$
be pairwise $\trianglelefteq$-incomparable, and let $l_{i}=l(\nu_{i}^{*})$.

\textbf{Claim.} $\left\{ \varphi_{l_{i}}\left(x,a_{\nu_{i}^{*}}\right):i<k\right\} $
is inconsistent.

Proof. By definition of $\nu_{i}^{*}$ and assumption on $\nu_{i}$'s
it follows that for any $i,i'<k$ the elements $\nu_{i}^{*}\restriction\left(l_{i}-1\right)$
and $\nu_{i'}^{*}\restriction\left(l_{i'}-1\right)$ are incomparable.
Then by Lemma \ref{s-IndiscTreeProp}(2) we
see that the sequences $\bar{a}_{i}=\left(a_{\nu_{i}^{*}\restriction\left(l_{i}-1\right)\widehat{\,\,}\langle j\rangle}:j\in\omega\right)$
are mutually indiscernible. But if $\left\{ \varphi_{l_{i}}\left(x,a_{\nu_{i}^{*}}\right):i<k\right\} $
was consistent, this would give us an $\inp$-pattern of depth $k$,
contrary to the assumption (as $\left\{ \varphi_{l_{i}}\left(x,a_{\nu_{i}^{*}\restriction\left(l_{i}-1\right)\widehat{\,\,}\langle j \rangle}\right):j\in\omega\right\} $
is $k_{l_{i}}$-inconsistent for every $i$).

Now using the claim it is easy to see that $\left\{ \varphi_{2l(\eta)}\left(x,a_{\eta^{*}}\right):\eta\in\omega^{<m}\right\} $
is an $(\sct,k)$-pattern of depth $m$. As $m$ was arbitrary, we conclude
that $\kappa_{\sct,k}^{n}\geq\aleph_{0}$.\end{proof}

\begin{prop} \label{prop: finding cdt,2 pattern}
Let $k<\omega$ be fixed. Assume that for any $n<\omega$ we have, in some fixed number of variables, an $\left(\sct,k\right)$-pattern of depth $n$. Then there are, in the same number of variables, $\left(\cdt,2\right)$-patterns of arbitrary finite depth.  
\end{prop}
\begin{proof}
Let $m\in\omega$ be arbitrary, and let $\left(a_{\eta}:\eta\in\omega^{<m\times m}\right),\left(\varphi_{i}\left(x,y_{i}\right):i<m\times m\right)$
be an $s$-indiscernible $(\sct,k)$-pattern - in particular this is a $\cdt$-pattern such that for $i < m \times m$, $\{\varphi_{i}(x;a_{\eta}) : l(\eta) = i\}$ is $k$-inconsistent.  

For $i<m$, consider 
$$
\Gamma_{i}\left(x\right)=\bigwedge_{l<m}\left(\varphi_{i\times m+l}\left(x,a_{0^{i \times m}\frown 0 \frown 0^{l-1}}\right)\land\varphi_{i \times m+l}\left(x,a_{0^{i \times m}\frown1\frown0^{l-1}}\right)\right).
$$
\textbf{Case 1.} $\Gamma_{i}\left(x\right)$ is consistent for some
$i<m$.

Obtain an $s$-indiscernible tree, using Lemma \ref{preservation}(1), by first taking the $2$-fold widening of $(a_{\eta})_{\eta \in \omega^{m \times m}}$ at level $i \times m+1$, then taking the restriction to $\{i \times m+l : l < m\}$.  Let $(\psi_{l} : l < m)$ be chosen so that 
$$
\psi_{l}\left(x,b_{0^{l}}\right)=\varphi_{i \times m+l}\left(x,a_{0^{i \times m} \frown 0 \frown 0^{l-1}}\right)\land\varphi_{i \times m+l}\left(x,a_{0^{i \times m} \frown 1 \frown 0^{l-1}}\right). 
$$
Then $\left(b_{\eta}:\eta\in\omega^{<m}\right),\left(\psi_{l}:l<m\right)$ is a $\cdt$-pattern of depth $m$ such that, for all $l < m$, $\{\psi_{l}(x;b_{\eta}) : l(\eta) = l\}$ is $\lfloor \frac{k}{2} \rfloor$-inconsistent.  

\textbf{Case 2.} $\Gamma_{i}\left(x\right)$ is inconsistent for every
$i<m$.

Using Lemma \ref{preservation}(1), obtain an $s$-indiscernible tree $(b_{\eta})_{\eta \in \omega^{<m}}$ by taking the $m$-fold elongation of $(a_{\eta})_{\eta \in \omega^{< m \times m}}$.  Let $(\psi_{l} : l < m)$ be chosen so that 
$$
\psi_{l}(x;b_{0^{l}}) = \bigwedge_{r < m} \varphi_{l \times m+r}(x;a_{0^{l \times m} \frown 0^{r}}).
$$
Then $(b_{\eta})_{\eta \in \omega^{<m}}$, $(\psi_{l} : l < m)$ is an $(\text{cdt},2)$-pattern.  

Repeating several times if necessary we conclude. \end{proof}


For $\kappa\leq\omega$, finding an $\sct$-pattern of depth $\kappa$
is equivalent to finding a $\left(\cdt,2\right)$-pattern of depth
$\kappa$.
\begin{lem} \label{lem: cdt,2 gives sct}
Let $\kappa\leq\omega$, and let $\left(a_{\eta}:\eta\in\omega^{<\kappa}\right),\left(\varphi_{i}\left(x,y_{i}\right):i<\kappa\right)$
be a $\left(\cdt,2\right)$-pattern (i.e. for every $\eta\in\omega^{<\kappa}$
the set $\left\{ \varphi_{l\left(\eta\right)+1}\left(x,a_{\eta\hat{\,}j}\right):j\in\omega\right\} $
is $2$-inconsistent). For $\eta\in\omega^{<\kappa}$ define $b_{\eta}=a_{\eta\restriction0}a_{\eta\restriction1}\ldots a_{\eta\restriction\left(l\left(\eta\right)-1\right)}a_{\eta}$
and $\psi_{i}\left(x;y_{i,0},\ldots,y_{i,i-1}\right)=\bigwedge_{j<i}\varphi_{j}\left(x,y_{j}\right)$.
Then $\left(b_{\eta}:\eta\in\omega^{<\kappa}\right),\left(\psi_{i}\left(x,\bar{y}_{i}\right):i<\kappa\right)$
is an $\sct$-pattern.\end{lem}
\begin{proof}
If $\eta\in\omega^{n}$ for $n<\kappa$, then the set $\left\{ \psi_{i}\left(x,b_{\eta\restriction i}\right):i<n\right\} $
contains only conjunctions of formulas from $\left\{ \varphi_{i}\left(x,a_{\eta\restriction i}\right):i<n\right\} $
which is consistent by assumption. On the other hand if $\eta_{1},\eta_{2}\in\omega^{<\kappa}$
are incomparable, let $\eta=\eta_{1}\land\eta_{2}$. Then $\psi_{l\left(\eta_{1}\right)}\left(x,b_{\eta_{1}}\right)$
implies $\varphi_{l\left(\eta\right)+1}\left(x,a_{\eta\hat{\,}\eta_{1}\left(l\left(\eta\right)+1\right)}\right)$
and $\psi_{l\left(\eta_{2}\right)}\left(x,b_{\eta_{2}}\right)$ implies
$\varphi_{l\left(\eta\right)+1}\left(x,a_{\eta\hat{\:}\eta_{2}\left(l\left(\eta\right)+1\right)}\right)$,
and these two implied formulas are inconsistent by assumption.\end{proof}


Combining Propositions \ref{prop: finding sct,k pattern} and \ref{prop: finding cdt,2 pattern} with Lemma \ref{lem: cdt,2 gives sct}, we have:
\begin{prop} \label{prop: aleph_0 case}
If $\kappa_{\cdt}^{n}\geq\aleph_{0}$, then either $\kappa_{\inp}^{n}\geq\aleph_{0}$
or $\kappa_{\sct}^{n}\geq\aleph_{0}$.\end{prop}
\begin{rem}
Inspecting the proof, we actually get the following bound:  $\kappa^n_{\sct} \geq (\frac{\kappa^n_{\cdt}}{2})^{\frac{1}{\kappa^n_{\inp}}}$.
\end{rem}

The next proposition is an analog of Proposition \ref{prop: finding cdt,2 pattern} for $\inp$-patterns. It is not used in this paper, but we include it for reference.

\begin{prop}
Let $k <\omega$ be fixed. Assume that for any $n< \omega$ we have, in some fixed number of free variables, an $\inp$-pattern of depth $n$ such that each row is $k$-inconsistent.
Then there are, in the same number of variables, $\inp$-patterns of arbitrary finite depths in which every row is $2$-inconsistent.\end{prop}
\begin{proof}
Let $m\in\omega$ be arbitrary, and let $\left(a_{i,j}\right)_{i<m\times m,j\in\omega},\left(\varphi_{i}\left(x,y_{i}\right)\right)_{i<m\times m}$
be an $\inp$-pattern with mutually indiscernible rows such that every
row is $k$-inconsistent. For $i<m$, consider $\Gamma_{i}\left(x\right)=\bigwedge_{i\times m\leq l<\left(i+1\right)\times m}\left(\varphi_{l}\left(x,a_{l,0}\right)\land\varphi_{l}\left(x,a_{l,1}\right)\right)$.

\textbf{Case 1.} $ $ $\Gamma_{i}\left(x\right)$ is consistent for
some $i<m$.

Then for $l<m$ we take $\psi_{l}\left(x,b_{l,0}\right)=\varphi_{i\times m+l}\left(x,a_{i\times m+l,0}\right)\land\varphi_{i\times m+l}\left(x,a_{i\times m+l,1}\right)$
and $b_{l,j}=a_{i\times m+l,2j}a_{i\times m+l,2j+1}$.

\textbf{Case 2.} $\Gamma_{i}\left(x\right)$ is inconsistent for every
$i<m$.

Then for $l<m$ we take $\psi_{l}\left(x,b_{l,0}\right)=\bigwedge_{r<m}\varphi_{l\times m+r}\left(x,a_{l\times m+r,0}\right)$
and $b_{l,j}=\left(a_{l\times m+r,j}:r<m\right)$.

It is easy to see that in each of the cases $\left(b_{i,j}\right)_{i<m,j<\omega},\left(\psi_{i}\left(x,y_{i}\right)\right)_{i<m}$
is an $\inp$-pattern of depth $m$, and moreover it is $\max\left\{ 2,\left\lceil \frac{k}{2}\right\rceil \right\} $-inconsistent
($\left\lceil \frac{k}{2}\right\rceil $-inconsistent in the first
case and $2$-inconsistent in the second case). As $m$ was arbitrary,
this shows that there are $\inp$-pattern of arbitrarily large finite
depth with $\max\left\{ 2,\left\lceil \frac{k}{2}\right\rceil \right\} $-inconsistent rows. Repeating the argument several times if necessary we conclude.\end{proof}

Now we consider the case of countably infinite patterns.

\begin{prop} \label{aleph_1 case}
\(\kappa^n_{cdt} \geq \aleph_{1}\) implies \(\kappa^n_{sct} \geq \aleph_{1}\). 
\end{prop}

\begin{proof}
Suppose \((\varphi_{i} : i < \omega)\), \((a_{\eta})_{\eta \in \omega^{<\omega}}\) is a cdt-pattern.  By replacing \(a_{\eta}\) with \(b_{\eta}  = (a_{\emptyset}, a_{\eta | 1}, \ldots, a_{\eta | l(\eta) -1}, a_{\eta})\) and \(\varphi_{i}(x;a_{\eta})\) by 
\[
\psi_{i}(x;b_{\eta}) := \bigwedge_{j \leq i} \varphi_{j}(x;a_{\eta | j }),
\]
if necessary, we may assume that if \(\nu \vartriangleleft \eta\), then
\[
\models (\forall x)[\varphi_{l(\eta)}(x;a_{\eta}) \to \varphi_{l(\nu)}(x;a_{\nu})].
\]
Then by replacing \((a_{\eta})_{\eta \in \omega^{<\omega}}\) by an \(s\)-indiscernible tree locally based on it, we may moreover assume the \((a_{\eta})_{\eta \in \omega^{<\omega}}\) are \(s\)-indiscernible by Fact \ref{modeling}.  

By induction, we will construct cdt-patterns \((\varphi_{i}^{n} : i < \omega)\), \((a^{n}_{\eta})_{\eta \in \omega^{<\omega}}\) so that 
\begin{enumerate}
\item \((a^{n}_{\eta})_{\eta \in \omega^{<\omega}}\) is \(s\)-indiscernible.
\item For all \(\eta \in \omega^{<n}\) and \(i < j\),
\[
\{\varphi^{n}_{l(\eta)+1}(x;a^{n}_{\eta \frown \langle i \rangle}), \varphi^{n}_{l(\eta)+1}(x;a^{n}_{\eta \frown \langle j \rangle})\}
\]
is inconsistent.
\item If \(\nu \vartriangleleft \eta\), then 
\[
\models (\forall x)[\varphi^{n}_{l(\eta)}(x;a^{n}_{\eta}) \to \varphi^{n}_{l(\nu)}(x;a^{n}_{\nu})].  
\]
\item For all \(\eta\), if \(n,n' \geq l(\eta)\), then \(a_{\eta}^{n} = a^{n'}_{\eta}\). For all \(m \leq m'\), \(\varphi^{m'}_{m} = \varphi^{m}_{m}\).  
\end{enumerate}
For the base case, let \(\varphi^{0}_{i} = \varphi_{i}\) for all \(i\) and \(a_{\eta}^{0} = a_{\eta}\) for all \(\eta\).  (1) is satisfied by assumption, (2) is vacuous, and (3) follows from the initial remarks above.  Now suppose we have constructed \((\varphi^{n}_{i} : i < \omega)\) and \((a^{n}_{\eta})_{\eta \in \omega^{<\omega}}\).  By definition of a cdt-pattern, there is a least \(k\geq 1\) so that 
\[
\bigcup_{i < 2^{k}} \{\varphi^{n}_{n+1+j}(x;a^{n}_{0^{n} \frown \langle i \rangle \frown 0^{j}}) : j < \omega \} 
\]
is inconsistent.  By compactness, there is \(N\) so that 
\begin{equation}\label{labelled}
 \bigcup_{i < 2^{k}} \{\varphi^{n}_{n+1+j}(x;a_{0^{n} \frown \langle i \rangle \frown 0^{j}}) : j < N\} 
\end{equation}
is inconsistent.  Let \((b_{\eta})_{\eta \in \omega^{<\omega}}\) be the \(N\)-fold stretch of \((a^{n})_{\eta \in \omega^{<\omega}}\) at level \(n\).  Let \((\psi_{i}(x;z_{i}): i < \omega)\) be defined as follows: for \(i \leq n\), \(z_{i} = y_{i}\) and \(\psi_{i}(x;z_{i}) = \varphi_{i}(x;y_{i})\).  Let \(z_{n+1} = (y_{n+1}, y_{n+2}, \ldots , y_{n+N})\) and 
\[
\psi_{n+1}(x;z_{n+1}) = \bigwedge_{j<N} \varphi^{n}_{n+1+j}(y;y_{n+1+j}).
\]
Finally, for \(i > n+1\), let \(z_{i} = y_{i+N-1}\) and \(\psi_{i}(x;z_{i}) = \varphi_{i+N-1}(x;y_{i+N-1})\).  By Lemma \ref{stretch}, \((b_{\eta})_{\eta \in \omega^{<\omega}}\) is an \(s\)-indiscernible tree and, by construction, \((\psi_{i}(x;z_{i}) : i < \omega)\), \((b_{\eta})_{\eta \in \omega^{<\omega}}\) is a cdt-pattern.  Moreover, this cdt-pattern satisfies
\begin{enumerate}
\setcounter{enumi}{4}
\item \(\{\psi_{n+1}(x;b_{0^{n} \frown \langle i \rangle}) : i < 2^{k}\}\) is inconsistent and 
\item \(\{\psi_{n+1+j}(x;b_{0^{n} \frown \langle i \rangle \frown 0^{j}}) : i < 2^{k-1}, j < \omega \} \cup \{\psi_{l}(x;b_{0^{l}}) : l < \omega \}\) is consistent.  
\end{enumerate} 
Condition (5) follows by the inconsistency \(\eqref{labelled}\) and the definition of \(\psi_{n+1}\).  To see \((6)\), we note that by the minimality of \(k\), 
\[
\{\psi_{n+1+j}(x;b_{0^{n} \frown \langle i \rangle \frown 0^{j}}) : i < 2^{k-1}, j < \omega \}
\]
is consistent.  By \((3)\) above and the definition of the \(\psi_{m}\), this establishes \((6)\).  

Let \((c_{\eta})_{\eta \in \omega^{<\omega}}\) be the \(2^{k-1}\)-fold widening of \((b_{\eta})_{\eta \in \omega^{<\omega}}\) at level \(n+1\).  Let \((\chi_{i}(x;w_{i}) : i < \omega)\) be defined as follows:  if \(i < n+1\), let \(w_{i} = z_{i}\) and \(\chi_{i}(x;w_{i}) = \psi_{i}(x;z_{i})\).  If \(i \geq n+1\), let \(w_{i} = (z_{i}^{0}, \ldots, z^{2^{k-1}-1}_{i})\)  a tuple of variables consisting of \(2^{k-1}\) copies of \(z_{i}\).  Then put 
\[
\chi_{i}(x;w_{i}) = \bigwedge_{j < 2^{k-1}} \psi_{i}(x;z_{i}^{j}).
\]
By Lemma \ref{widening}, \((c_{\eta})_{\eta \in \omega^{<\omega}}\) is \(s\)-indiscernible and, by construction, \((\chi_{i}(x;w_{i}) : i < \omega)\), \((c_{\eta})_{\eta \in \omega^{<\omega}}\) is a cdt-pattern and, moreover, if \(i \neq j\) 
\[
\{\chi_{n+1}(x;c_{0^{n} \frown \langle i \rangle}) , \chi_{n+1}(x;c_{0^{n} \frown \langle j \rangle})\}
\]
is inconsistent.  For all \(m < \omega\) and \(\eta \in \omega^{<\omega}\), define \(\varphi^{n+1}_{m} = \xi_{m}\) and \(a^{n+1}_{\eta} = c_{\eta}\).  We have satisfied requirements (1)-(3) and since our construction did not modify the formulas and parameters with level at most \(n\), the construction never injures requirement (4).  

Finally, define a cdt-pattern \((\varphi^{\infty}_{n} : n < \omega)\), \((a^{\infty}_{\eta})_{\eta \in \omega^{<\omega}}\) by \(\varphi^{\infty}_{n} = \varphi^{n}_{n}\) and \(a^{\infty}_{\eta} = a^{l(\eta)}_{\eta}\).  Our construction gives
\begin{enumerate}
\setcounter{enumi}{6}
\item \((a^{\infty}_{\eta})_{\eta \in \omega^{<\omega}}\) is \(s\)-indiscernible.
\item If \(\eta \in \omega^{\omega}\), \(\{\varphi^{\infty}(x;a^{\infty}_{\eta | n}) : n < \omega \}\) is consistent.
\item If \(\nu \vartriangleleft \eta\), then $\models (\forall x)[\varphi^{\infty}_{l(\eta)}(x;a^{\infty}_{\eta}) \to \varphi^{\infty}_{l(\nu)}(x;a^{\infty}_{\nu})].$
\item For all \(n\), and \(i \neq j\) $\{\varphi^{\infty}_{n+1}(x;a^{\infty}_{0^{n} \frown \langle i \rangle}) , \varphi^{\infty}_{n+1}(x;a^{\infty}_{0^{n} \frown \langle j \rangle}) \}$ is inconsistent.  
\end{enumerate}
By \(s\)-indiscernibility, (9) and (10) imply that if \(\eta \perp \nu\), then
\[
\{\varphi^{\infty}_{l(\eta)}(x;a^{\infty}_{\eta}), \varphi^{\infty}_{l(\nu)}(x;a^{\infty}_{\nu})\}
\]
is inconsistent.  This shows \((\varphi^{\infty}_{n} : n < \omega )\) and \((a^{\infty}_{\eta})_{\eta \in \omega^{<\omega}}\) form an \(sct\)-pattern.  We have thus shown \(\kappa^n_{sct} \geq \aleph_{1}\).
\end{proof}

We obtain the main theorem of this section.
\begin{thm}
If \(T\) is countable, then $\kappa_{\text{cdt}}(T) = \kappa_{\sct}(T) + \kappa_{\inp}(T)$.  Moreover, $\kappa^{n}_{\cdt}(T) = \kappa^{n}_{\sct}(T) + \kappa^{n}_{\inp}(T)$, provided $\kappa^{n}_{\text{cdt}}(T)$ is infinite.  
\end{thm}

\begin{proof}
By Observation \ref{basicobs}, $\kappa^{n}_{\text{cdt}}(T) \geq n$ for any $T$ and $\kappa_{\text{cdt}}(T) > |T|^{+}$ if and only if $\kappa_{\text{cdt}}(T) = \infty$. It follows that, for countable theories, the possible values of $\kappa_{\cdt}(T)$, and the only possible infinite values of \(\kappa^{n}_{\text{cdt}}(T)\), are $\aleph_{0}$, $\aleph_{1}$, and \(\infty\).  The case of \(\aleph_{0}\) is treated in Proposition \ref{prop: aleph_0 case}, \(\aleph_{1}\) is handled by Proposition \ref{aleph_1 case}, and for \(\infty\) the result follows from Shelah's theorem (Fact \ref{dichotomy}).
\end{proof}

\section{$\TP_1$ and weak $k-\TP_1$}\label{TP1}

Say that a subset \(\{\eta_{i} : i < k\} \subseteq \omega^{<\omega}\) is a collection of \emph{distant siblings} if given \(i \neq i'\), \(j \neq j'\), all of which are \(<k\), \(\eta_{i} \wedge \eta_{i'} = \eta_{j} \wedge \eta_{j'}\).  

\begin{defn}\label{weak}
Fix \(k \geq 2\).  
\begin{enumerate}
\item The formula \(\varphi(x;y)\) has $\SOP_2$ if there is a collection of tuples \((a_{\eta})_{\eta \in 2^{<\omega}}\) satisfying the following.
\begin{enumerate}
\item For all \(\eta \in 2^{\omega}\), \(\{\varphi(x;a_{\eta | \alpha}) : \alpha < \omega\}\) is consistent.
\item   If \(\eta,\nu \in 2^{<\omega}\) and \(\eta \perp \nu\), then \(\{\varphi(x;a_{\eta}), \varphi(x;a_{\nu})\}\) is inconsistent.  
\end{enumerate}
\item The formula \(\varphi(x;y)\) has \emph{weak k-TP}\(_{1}\) if there is a collection of tuples \((a_{\eta})_{\eta \in \omega^{<\omega}}\) satisfying the following.
\begin{enumerate}
\item For all \(\eta \in \omega^{\omega}\), \(\{\varphi(x;a_{\eta | \alpha}) : \alpha < \omega\}\) is consistent.
\item If \(\{\eta_{i} : i < k\} \subseteq \omega^{<\omega}\) is a collection of distinct distant siblings, then \(\{\varphi(x;a_{\eta_{i}}) : i < k\}\) is inconsistent.  
\end{enumerate}
\item The formula \(\varphi(x;y)\) has \emph{k-TP}\(_{1}\) if there is  a collection of tuples \((a_{\eta})_{\eta \in \omega^{<\omega}}\) satisfying the following.
\begin{enumerate}
\item For all \(\eta \in \omega^{\omega}\), \(\{\varphi(x;a_{\eta | \alpha}) : \alpha < \omega\}\) is consistent.
\item If \(\{\eta_{i} : i < k\} \subseteq \omega^{<\omega}\) is a collection of distinct pairwise incomparable nodes, then \(\{\varphi(x;a_{\eta_{i}}) : i < k\}\) is inconsistent.  
\end{enumerate}
\item The theory \(T\) has either of the above properties if some formula does.  
\end{enumerate}
\end{defn}

We remark that $\TP_1$ is equivalent to $\SOP_2$ in a strong way:

\begin{fact} \label{SOP2iffTP1}
If a theory has $\TP_1$ witnessed by a formula \(\varphi\), then the theory also has $\SOP_2$ witnessed by the same formula, and vice versa.
\end{fact}
We recall the argument from \cite{Adler}.  Suppose \(\varphi(x;y)\) witnesses $\SOP_2$ with respect to the tree of parameters \((b_{\eta})_{\eta \in 2^{<\omega}}\).  Define a map \(h: \omega^{<\omega} \to 2^{<\omega}\) recursively by \(h(\emptyset) = \emptyset\) and \(h(\beta \frown \langle i \rangle) = h(\beta) \frown 1^{i} \frown 0\), where \(1^{i}\) denotes the all \(1\)'s sequence of length \(i\).  It is straightforward to check that \(\varphi(x;y)\) witnesses $\TP_1$ with respect to the parameters \((b_{h(\eta)})_{\eta \in \omega^{<\omega}}\). The converse is obvious. Although $\SOP_2$ and $\TP_1$ are equivalent, it will be important for us to notationally distinguish them, as various combinatorial constructions are simplified by a judicious choice of the index set.  

In \cite{KimKimNTP1}, Kim and Kim show that $k$-TP$_1$ is equivalent to $\TP_1$ for all $k\geq 2$, but the questions of whether weak $k$-TP$_1$ is equivalent to $\TP_1$ was left unresolved.  Using strongly indiscernible trees, we settle this, as well as show that $\TP_1$ is always witnessed by a formula in a single free variable. 

\subsection{Finding and manipulating indiscernible witnesses}

\begin{lem} \label{StronglyIndiscSOP2}

 \begin{enumerate}
  \item If \(T\) has weak k-TP\(_{1}\) witnessed by \(\varphi(x;y)\) then there is a strongly indiscernible tree \((a_{\eta})_{\eta \in \omega^{<\omega}}\) witnessing this.
  \item If \(\varphi(x;y)\) has $\TP_1$ then there is a strongly indiscernible tree witnessing this.  
\item If $\varphi(x,y)$ has $\SOP_2$, then there is a strongly indiscernible tree \((a_{\eta})_{\eta \in 2^{< \omega}}\) witnessing this.  
 \end{enumerate}
\end{lem}

\begin{proof}
(1)  This was observed in \cite{TakeuchiTsuboi}, but we sketch a proof here for completeness.  Let \((b_{\eta})_{\eta \in \omega^{<\omega}}\) be a tree of tuples with respect to which \(\varphi(x;y)\) witnesses weak $k$-TP\(_{1}\).  Let \((a_{\eta})_{\eta \in \omega^{<\omega}}\) be locally based on the tree \((b_{\eta})_{\eta \in \omega^{<\omega}}\).  Suppose \(\eta_{0}, \ldots, \eta_{n-1} \in \omega^{<\omega}\) lie along a path and let \(\psi(y_{0}, \ldots, y_{n-1})\) denote the formula \((\exists x)\bigwedge_{i < n} \varphi(x;y_{i})\).  Then there are \(\nu_{0}, \ldots, \nu_{n-1} \in \omega^{<\omega}\) so that 
\[
\text{qftp}_{L_{0}}(\eta_{0}, \ldots, \eta_{n-1}) = \text{qftp}_{L_{0}}(\nu_{0}, \ldots, \nu_{n-1})
\]
and 
\[
\text{tp}_{\psi}(a_{\eta_{0}}, \ldots, a_{\eta_{n-1}}) = \text{tp}_{\psi}(b_{\nu_{0}}, \ldots, b_{\nu_{n-1}}).
\]
The first equality implies that \(\nu_{0}, \ldots, \nu_{n-1}\) all lie along a path so \(\{\varphi(x;b_{\nu_{i}}) : i < n\}\) is consistent.  By the second equality, \(\{\varphi(x;a_{\eta_{i}}) : i < n\}\) is consistent.  By compactness, this shows that all paths are consistent.  Showing that any $k$ distinct distant siblings remain inconsistent is similar.  So \(\varphi(x;y)\) witnesses weak $k$-TP\(_{1}\) with respect to the tree \((a_{\eta})_{\eta \in \omega^{<\omega}}\).  

(2)  This follows from (1) as weak \(2\)-TP\(_{1}\) and $\TP_1$ are the same.  

(3)  By Fact \ref{SOP2iffTP1}, $\varphi(x,y)$ has $\TP_1$. Now by (2), we may find a strongly indiscernible tree \((a_{\eta})_{\eta \in \omega^{<\omega}}\) such that \(\varphi\) witnesses $\TP_1$ with respect to \((a_{\eta})_{\eta \in \omega^{<\omega}}\). Making the identification \(2^{<\omega} = \{\eta \in \omega^{<\omega} : \eta(k) \in \{0,1\} \mbox{ for all } k < l(\eta)\}\), it is easy to see that \((2^{<\omega}, \vartriangleleft, <_{lex}, \wedge)\) is an \(L_{0}\)-substructure of \((\omega^{<\omega}, \unlhd, <_{lex}, \wedge)\) since \(2^{<\omega}\) is closed under the \(\wedge\)-function and all the symbols in \(L_{0}\) acquire their natural interpretation on \(2^{<\omega}\) via restriction from \(\omega^{<\omega}\).  It follows that if \(\eta_{0}, \ldots, \eta_{n-1}\) and \(\nu_{0}, \ldots, \nu_{n-1}\) are two sequences from \(2^{<\omega}\) with 
\[
\text{qftp}_{L_{0}}(\eta_{0}, \ldots, \eta_{n-1}) = \text{qftp}_{L_{0}}(\nu_{0}, \ldots, \nu_{n-1})
\]
in \(2^{<\omega}\), then this equality also holds in \(\omega^{<\omega}\) and hence
\[
\text{tp}(a_{\eta_{0}}, \ldots, a_{\eta_{n-1}}) = \text{tp}(a_{\nu_{0}}, \ldots, a_{\nu_{n-1}}),
\]
so \((a_{\eta})_{\eta \in 2^{<\omega}}\) is strongly indiscernible.  Moreover, paths in \(2^{\omega}\) are paths also in \(\omega^{\omega}\) and incomparables in \(2^{<\omega}\) remain incomparable when considered as elements in \(\omega^{<\omega}\) so it is clear that \(\varphi(x;y)\) will witness $\SOP_2$ with respect to \((a_{\eta})_{\eta \in 2^{<\omega}}\).  
\end{proof}

\begin{rem}
We \emph{aren't} making the (ostensibly) stronger claim that if \(\varphi(x;y)\) witnesses $\SOP_2$ with respect to the tree \((b_{\eta})_{\eta \in 2^{<\omega}}\) then there is a strongly indiscernible tree \((a_{\eta})_{\eta \in 2^{<\omega}}\) \emph{based} on it --- the proof of the existence of a strongly indiscernible tree witness involved going through $\TP_1$ and then restricting.  
\end{rem}

\begin{lem}\label{collision}
 \begin{enumerate}
  \item  If \((a_{\eta})_{\eta \in \omega^{<\omega}}\) is a strongly indiscernible tree and \(\varphi(x;y)\) is a formula so that for some \(\eta \in \omega^{\omega}\), \(\{\varphi(x;a_{\eta | n}) : n < \omega\}\) is consistent and for some \(\xi \in \omega^{< \omega}\), \(\{\varphi(x;a_{\xi \frown 0}), \varphi(x; a_{\xi \frown 1})\}\) is inconsistent, then \(T\) has $\TP_1$.
\item If \((a_{\eta})_{\eta \in 2^{<\omega}}\) is a strongly indiscernible tree and \(\varphi(x;y)\) is a formula so that for some \(\eta \in 2^{\omega}\), \(\{\varphi(x;a_{\eta | n}) : n < \omega\}\) is consistent and for some \(\eta \in 2^{< \omega}\), \(\{\varphi(x;a_{\eta \frown 0}), \varphi(x; a_{\eta \frown 1})\}\) is inconsistent, then \(T\) has $\SOP_2$.
\end{enumerate}
\end{lem}

\begin{proof}
Both parts are immediate by Lemma \ref{StronglyIndiscTreeProp}, (1) and (2).
\end{proof}

\begin{lem}\label{pathcollapse}{(Path Collapse)}
Suppose \(\kappa\) is an infinite cardinal, \((a_{\eta})_{\eta \in 2^{<\kappa}}\) is a tree strongly indiscernible over a set of parameters \(C\) and, moreover, \((a_{0^{\alpha}}: 0 < \alpha < \omega)\) is indiscernible over \(cC\).  Let 
\[
p(y; \overline{z}) = \tp(c;(a_{0 \frown 0^{\gamma}} : \gamma < \kappa)/C).
\]
Then if 
\[
p(y;(a_{0 \frown 0^{\gamma}})_{\gamma < \kappa}) \cup p(y;(a_{1 \frown 0^{\gamma}})_{\gamma < \kappa})
\]
is not consistent, then \(T\) has $\SOP_2$, witnessed by a formula with free variables \(y\).    
\end{lem}

\begin{proof}
We may add \(C\) to the language, so assume \(C = \emptyset\).  With \(p\) defined as above, suppose 
\[
p(y;(a_{0 \frown 0^{\gamma}} : \gamma < \kappa)) \cup p(y;(a_{1 \frown 0^{\gamma}} : \gamma < \kappa))
\]
is inconsistent.  Then by indiscernibility and compactness, there is a formula \(\psi \) and $n < \omega$ so that 
\[
\{ \psi(y;a_{0}, \ldots, a_{0 \frown 0^{n-1}}) \} \cup \{ \psi(y;a_{1}, a_{10}, \ldots, a_{1 \frown 0^{n-1}}) \}
\]
is inconsistent.  Let \((b_{\eta})_{\eta \in 2^{<\kappa}}\) denote the \(n\)-fold elongation of \((a_{\eta})_{\eta \in 2^{<\kappa}}\).  By Lemma \ref{preservation}, $(b_\eta : \eta \in 2^{< \kappa})$ is strongly indiscernible. Since \(c \models \{\psi(y;b_{0^{\alpha}}) : \alpha < \kappa\}\) and $\psi(y;b_0) \land \psi(y;b_1)$ is inconsistent (by strong indiscernibility), by Lemma \ref{collision}, \(\psi\) witnesses $\SOP_2$.
\end{proof}

\begin{rem}\label{noemptyset}
It is significant that the type $p$ does \emph{not} contain $a_{\emptyset}$ as a parameter.  As $b_{0}$ and $b_{1}$ are incomparable and $\psi(x;b_{0})$ and $\psi(x;b_{1})$ are inconsistent, we can conclude that $\psi(x;b_{\eta})$ and $\psi(x;b_{\nu})$ are inconsistent for all incomparable $\eta, \nu$ by strong indiscernibility.  But, for example, strong indiscernibility does not guarantee $b_{0\frown 0}b_{0\frown 1}$ has the same type as $b_{0}b_{1}$ over $a_{\emptyset}$ as $0 \wedge 1 = \emptyset$ while $ 0^{n-1} \frown 0 \wedge 0^{n-1} \frown 1 = 0^{n-1}$.  
\end{rem}

We now give two applications of the path-collapse lemma.  

\subsection{Weak $k-\TP_1$}

\begin{thm}
Given \(k \geq 2\), \(T\) has weak \(k\)-TP\(_{1}\) if and only if \(T\) has $\TP_1$.  
\end{thm}

\begin{proof}
We will show that if \(T\) has weak \(k\)-TP\(_{1}\), then \(T\) has $\SOP_2$.  Let \(\varphi(x;y)\) witness weak \(k\)-TP\(_{1}\) with respect to the strongly indiscernible tree \((a_{\eta})_{\eta \in \omega^{<\omega}}\).  Let \(n\) be maximal so that 
\[
\{\varphi(x;a_{\langle i \rangle \frown 0^{\alpha}}) : i < n, \alpha < \omega\}
\]
is consistent.  By definition of weak \(k\)-TP\(_{1}\), \(n\) is at least 1 and at most \(k-1\).  Let \(C = \{a_{\langle i \rangle \frown 0^{\alpha}} : i < n - 1, \alpha < \omega\}\) (and put \(C = \emptyset\) in the case that \(n = 1\)).  Given \(\eta \in \omega^{<\omega}\), let \(\hat{\eta}\) be defined by 
\[
\hat{\eta}(i)  = \left\{
\begin{matrix}
\eta(i) + n-1 & \text{ if } i = 0 \\
\eta(i) & \text{ otherwise},
\end{matrix}
\right.
\]
for all \(i < l(\eta)\).  The tree \((b_{\eta})_{\eta \in \omega^{<\omega}}\) defined by \(b_{\eta} = a_{\hat{\eta}}\) is strongly indiscernible over \(C\).  By choice of \(n\), 
\[
\{\varphi(x;a_{\langle i \rangle \frown 0^{\alpha}}) : i < n, \alpha < \omega\}
\]
is consistent, so let \(c\) realize it.   By compactness, Ramsey, and automorphism, we may assume \((b_{0^{\alpha}} : 0 < \alpha < \omega)\) (i.e. \((a_{\langle n-1 \rangle \frown 0^{\alpha}} : \alpha < \omega))\) is indiscernible over \(c\).  Letting the type \(p\) be defined by
\[
p(y; \overline{z}) = \text{tp}(c;(b_{0 \frown 0^{\alpha}} : \alpha < \alpha)/C),
\]
and unravelling definitions, we see that the type
\[
p(y;(b_{0 \frown 0^{\alpha}} : \alpha < \omega)) \cup p(y;(b_{1 \frown 0^{\alpha}} : \alpha < \omega))
\]
implies \(\{\varphi(x;a_{\langle i \rangle \frown 0^{\alpha}}) : i < n+1, \alpha < \omega\}\) and is therefore inconsistent by the choice of \(n\).  By path-collapse, we've shown that \(T\) has $\SOP_2$, completing one direction.  The other direction is obvious.  
\end{proof}

\subsection{Reducing to one variable}

\begin{prop}
Suppose \(T\) witnesses $\SOP_2$ via \(\varphi(x,y;z)\).  Then there is a formula \(\varphi_{0}(x;v)\) with free variables \(x\) and parameter variables \(v\), or a formula \(\varphi_{1}(y;w)\) with free variables \(y\) and parameter variables \(w\) so that one of \(\varphi_{0}\) and \(\varphi_{1}\) witness $\SOP_2$.
\end{prop}

\begin{proof}
Let \(\varphi(x,y;z)\) witness $\SOP_2$ with respect to the strongly indiscernible tree \((a_{\eta})_{\eta \in 2^{<\omega}}\).  The first path is consistent and it is an indiscernible sequence so it follows that there is some \((c,c_{0}) \models \{\varphi(x,y;a_{0^{\alpha}}) : \alpha < \omega\}\) and such that moreover \((a_{0^{\alpha}} : \alpha < \omega)\) is indiscernible over \(c_{0}\) (by Ramsey, automorphism, and compactness).

Define the function \(h: 2^{<\omega} \to 2^{<\omega}\) recursively by \(h(\emptyset) = \emptyset\) and \(h(\eta \frown \langle i\rangle) = h(\eta) \frown 0 \frown \langle i \rangle\).  Define the tree \((b_{\eta})_{\eta \in 2^{<\omega}}\) by \(b_{\eta} = a_{h(\eta)}\).  It is proved in Lemma \ref{thehfunction}(1) that $(b_{\eta})_{\eta \in 2^{<\omega}}$ is a strongly indiscernible tree.  For each \(n\), define a map \(h_{n}: 2^{<\omega} \to 2^{<\omega}\) by
\[
h_{n}(\eta) =
\left\{
\begin{matrix}
h(\eta) & \text{ if } l(\eta) \leq n \\
h(\nu) \frown \xi & \text{ if } \eta = \nu \frown \xi, l(\nu) = n.  
\end{matrix}
\right.
\]
By Lemma \ref{thehfunction}(2), the tree \((d_{n,\eta})_{\eta \in 2^{<\omega}}\) defined by \(d_{n,\eta} = a_{h_{n}(\eta)}\) is strongly indiscernible as well.  Moreover, as paths in \((b_{\eta})_{\eta \in 2^{<\omega}}\) and \((d_{n,\eta})_{\eta \in 2^{<\omega}}\) are contained in paths in \((a_{\eta})_{\eta \in 2^{<\omega}}\) and incomparable elements in these trees correspond to incomparable elements in \((a_{\eta})_{\eta \in 2^{<\omega}}\), \(\varphi\) witnesses $\SOP_2$ with respect to these trees of parameters as well.  

Assume that no formula in the variable \(y\) has $\SOP_2$.  By induction, we will choose \(c_{n}\) so that 
\begin{equation}
   \{\varphi(x,c_{n};d_{n,\eta |m}) : m < n\} \cup \{\varphi (x,c_{n};d_{n,\eta \frown 0^{\alpha}}) : \alpha < \omega\}  \tag{*}
  \end{equation}
  is consistent for every \(\eta \in 2^{\leq n}\).  
  
For this, consider \((d^{(n)}_{n,\eta})_{\eta \in 2^{<\omega}}\), the \(n\)th-fattening of \((d_{n,\eta})\), and let \(C_{n} = (d_{n,\eta} : \eta \in 2^{<n})\).  By induction we show:

\textbf{Claim.}  There is \(c_{n+1}\) such that \(\left((d_{n+1,0^{\alpha}}^{(n+1)}) : \alpha < \omega\right)\) is indiscernible over \(c_{n+1}C_{n}\) and
\[
c_{n}\left(d_{n,0 \frown 0 \frown 0^{\alpha}}^{(n)}\right) \equiv_{d_{n,\emptyset}^{(n)}C_{n}} c_{n+1} \left(d_{n,0 \frown 0 \frown 0^{\alpha}}^{(n)}\right) \equiv_{d_{n,\emptyset}^{(n)}C_{n}} c_{n+1} \left(d_{n,0 \frown 1 \frown 0^{\alpha}}^{(n)} \right).
\]
Note that \(d^{(n)}_{n,\emptyset}C_{n} = C_{n+1}\).  

\emph{Proof:}  The base case is above.  Let $$p_{n}(y, \overline{z}) = \text{tp}\left(c_{n}, (d^{(n)}_{n,0 \frown 0 \frown 0^{\alpha}} : \alpha < \omega)/(d_{n,\emptyset})^{(n)}C_{n}\right).$$ By the path-collapse lemma, 
\[
p_{n}\left(y, \left((d^{(n)}_{n,0\frown 0 \frown 0^{\alpha}}) : \alpha < \omega\right)\right) \cup p_{n}\left(y, \left((d^{(n)}_{n,0\frown 1 \frown 0^{\alpha}}) : \alpha < \omega\right)\right)
\]
is consistent.  Let \(c_{n+1}\) realize it.  Moreover, as $$\left(d_{n,0 \frown 0 \frown 0^{\alpha}}^{(n)}, d_{n,0 \frown 1 \frown 0^{\alpha}}^{(n)}\right)_{\alpha < \omega} = \left(d_{n+1,0^{\alpha}}^{(n+1)}\right)_{\alpha < \omega}$$ is an indiscernible sequence, by Ramsey, automorphism, and compactness we may assume that it is indiscernible over \(c_{n+1}C_{n}\).  This shows (*).

By the definition of the trees \((d_{n,\eta})_{\eta \in 2^{<\omega}}\), we have shown that 
$$ \{\varphi(x,c_{n};b_{\eta |m}) : m < n\} \cup \{\varphi (x,c_{n};b_{\eta \frown 0^{\alpha}}) : \alpha < \omega\}$$  is consistent for each \(n\) and \(\eta \in 2^{\leq n}\).  
  By compactness, we can find one \(c\) which works for all possible paths in \(2^{\omega}\) simultaneously, giving a tree \((c, b_{\eta})_{\eta \in 2^{<\omega}}\) witnessing $\SOP_2$ for \(\varphi(x;y,z)\).
\end{proof}

\begin{rem}
The necessity of defining the trees $(b_{\eta})_{\eta \in 2^{<\omega}}$ and $(d_{n,\eta})_{\eta \in 2^{<\omega}}$ via $h$ and $h_{n}$, respectively, stems from a technical obstacle in applying the path-collapse lemma:  starting with the tree $(a_{\eta})_{\eta \in 2^{<\omega}}$, we cannot apply the path collapse lemma directly to the type 
$$
q(y;(a_{0^{\alpha}} : \alpha < \omega)) = \text{tp}(c_{0}/(a_{0^{\alpha}} : \alpha < \omega)),
$$
as this type has $a_{\emptyset}$ as a parameter (see Remark \ref{noemptyset} above).  This is corrected by the offset functions $h$ and $h_{n}$, allowing us to apply the path-collapse lemma `higher' in the tree, where the parameters of interest are indiscernible over what we have constructed so far.    
\end{rem}

\begin{cor}\label{1var}
 \begin{enumerate}
\item \(T\) has $\SOP_2$ if and only if there is some formula in a single free variable witnessing this
  \item \(T\) has $\TP_1$ if and only if there is some formula in a single free variable witnessing this
 \end{enumerate}
\end{cor}

At this point it is natural to ask if $\kappa_{\sct}^{1}=\kappa_{\sct}^{n}$
holds for arbitrary $n$, at least for countable theories. Corollary \ref{1var}
resolves the case of $\infty$, and we remark that the case of $\aleph_{1}$
follows from a standard argument in simplicity theory.
\begin{prop}
\label{prop: trans kappa cdt}Any theory satisfies $\kappa_{\cdt}^{1}=\kappa_{\cdt}^{n}$,
for all $n\in\omega$.\end{prop}
\begin{proof}
The following are equivalent (see e.g. \cite[Proposition 3.8]{casanovas2011simple}).
\begin{enumerate}
\item $\kappa_{\cdt}^{n}\leq\kappa$.
\item For any type $p\left(x\right)\in S_{n}\left(A\right)$, there is some
$A_{0}\subseteq A$ such that $\left|A_{0}\right|<\kappa$ and $p$
does not divide over $A_{0}$.
\end{enumerate}
Clearly $\kappa_{\cdt}^{n}\geq\kappa_{\cdt}^{1}$. Assume now that
$\kappa_{\cdt}^{1}\leq\kappa$ for some $\kappa$. We show by induction
that (2) above holds for all $n$ with respect to $\kappa$. Given
$a_{1}\ldots a_{n}a_{n+1}$ and $A$, it follows by the inductive
assumption that $a_{1}\ldots a_{n}\ind_{A_{0}}A$ for some $A_{0}\subseteq A$
with $\left|A_{1}\right|<\kappa$ and $a_{n+1}\ind_{A_{1}a_{1}\ldots a_{n}}Aa_{1}\ldots a_{n}$
for some $A_{1}\subseteq A$ with $\left|A_{1}\right|<\kappa$. Combined
this implies (by left transitivity and right base monotonicity of
dividing in arbitrary theories, see e.g. \cite[Section 2]{chernikov2012forking})
that $a_{1}\ldots a_{n}a_{n+1}\ind_{A_{0}A_{1}}A$ and $\left|A_{0}\cup A_{1}\right|<\kappa$.\end{proof}
\begin{cor}
If $\kappa_{\sct}^{n}\geq\aleph_{1}$ then $\kappa_{\sct}^{1}\geq\aleph_{1}$.\end{cor}
\begin{proof}
By Proposition \ref{aleph_1 case}, it is enough to show that $\kappa_{\cdt}^{1}\geq\aleph_{1}$,
which follows by assumption and Proposition \ref{prop: trans kappa cdt}.
\end{proof}
The case of $\aleph_{0}$ appears to involve more complicated combinatorics
and we leave it for future work.

\section{Independence and amalgamation in $\NSOP_1$ theories}

We recall the definition of $\SOP_1$ from \cite{ShUs:844}:

\begin{defn}
 A formula \(\varphi(x;y)\) exemplifies $\SOP_1$ if and only if there are \((a_{\eta})_{\eta \in 2^{<\omega}}\) so that 
\begin{itemize}
 \item For all \(\eta \in 2^{\omega}\), \(\{\varphi(x;a_{\eta|n}) : n < \omega\}\) is consistent,
\item If \(\eta \frown 0 \trianglelefteq \nu \in 2^{<\omega}\), then \(\{\varphi(x;a_{\eta \frown 1}), \varphi(x;a_{\nu})\}\) is inconsistent.
\end{itemize}
\end{defn}

Given an array \((c_{i,j})_{i < \omega, j < 2}\), write \(\overline{c}_{i} = (c_{i,0}, c_{i,1})\) and \(\overline{c}_{<i}\) for \((\overline{c}_{j})_{j < i}\).  

\begin{lem}\label{automorphisms}
Suppose \((c_{i,j})_{i < \omega, j < 2}\) is an array and \(\varphi(x;y)\) is a formula over \(C\) with 
\begin{enumerate}
\item For all \(i < \omega\), \(c_{i,0} \equiv_{C \overline{c}_{<i}} c_{i,1}\);
\item \(\{\varphi(x;c_{i,0}) : i < \omega\}\) is consistent; 
\item \(j \leq i \implies \{\varphi(x;c_{i,0}), \varphi(x;c_{j,1})\}\) is inconsistent,
\end{enumerate}
then \(T\) is $\SOP_1$.  
\end{lem}

\begin{proof}
For each $n$, define a subtree \(T_{n}\) of $2^{<\omega}$ by
$$
T_{n} = \{\eta \frown 0^{\alpha}: \eta \in 2^{\leq n}, \alpha < \omega\} \cup \{\eta \frown 0^{\alpha} \frown 1: \eta \in 2^{\leq n}, \alpha < \omega\}.
$$
Let $P(T_{n}) \subseteq 2^{\omega}$ be the set of infinite branches of $T_{n}$.  Namely, 
$$
P(T_{n}) = \{\eta \frown 0^{\omega}: \eta \in 2^{\leq n}\}.
$$
As a first step, by induction on $n$ we build an ascending sequence of trees $(l_{\eta},r_{\eta})_{\eta \in T_{n}}$, so that:
\begin{enumerate}
\item if $\eta \in P(T_{n})$, $(l_{\eta | \alpha},r_{\eta | \alpha})_{\alpha < \omega} \equiv_{C} (c_{\alpha, 0},c_{\alpha,1})_{\alpha < \omega}$,
\item if $\eta \frown 0 \in T_{n}$ then $r_{\eta \frown 0} = l_{\eta \frown 1}$,
\item if $\eta \in 2^{\leq n}$ then $(l_{\eta \frown 0}, r_{\eta \frown 0}) \equiv_{C l_{\unlhd \eta} r_{\unlhd \eta}} (l_{\eta \frown 1},r_{\eta \frown 1})$.  
\end{enumerate}
For the $n = 0$ case, define $l_{0^{\alpha}}  = c_{\alpha,0}$, $r_{0^{\alpha}} =  c_{\alpha, 1}$ and $l_{0^{\alpha} \frown 1} = r_{0^{\alpha} \frown 0}$ for all $\alpha < \omega$.  For each $\alpha < \omega$, we can choose $\sigma_{\alpha} \in \text{Aut}(\mathbb{M}/C \overline{c}_{< \alpha})$ such that $\sigma_{\alpha}(c_{\alpha, 0}) = c_{\alpha, 1}$.  Let $r_{0^{\alpha}\frown 1} = \sigma_{\alpha+ 1}(c_{\alpha + 1,1}) = \sigma_{\alpha + 1}(r_{0^{\alpha} \frown 0})$.  This defines $(l_{\eta},r_{\eta})_{\eta \in T_{0}}$ satisfying (1)-(3).  

Now by induction suppose $(l_{\eta}, r_{\eta})_{\eta \in T_{n}}$ has been defined.  Suppose $\eta \in P(T_{n+1}) \setminus P(T_{n})$.  Then there is $\nu \in 2^{\leq n}$ so that $\eta = \nu \frown 1 \frown 0^{\omega}$.  Then $\nu \frown 1 \in T_{n}$ and, by induction, 
$$
(l_{\nu \frown 0}, r_{\nu \frown 0}) \equiv_{C l_{\unlhd \nu} r_{\unlhd \nu}} (l_{\nu \frown 1},r_{\nu \frown 1})
$$
and $r_{\nu \frown 0} = l_{\nu \frown 1}$.  Choose an automorphism $\sigma \in \text{Aut}(\mathbb{M}/C l_{\unlhd \nu} r_{\unlhd \nu})$ such that $\sigma (l_{\nu \frown 0}, r_{\nu \frown 0}) = l_{\nu \frown 1}, r_{\nu \frown 1}$.  Then define 
$$(l_{\nu \frown 1 \frown 0^{\alpha}}, r_{\nu \frown 1 \frown 0^{\alpha}}) = \sigma (l_{\nu \frown 0 \frown 0^{\alpha}},r_{\nu \frown 0 \frown 0^{\alpha}}) \textrm{ and }$$
$$(l_{\nu \frown 1 \frown 0^{\alpha} \frown 1}, r_{\nu \frown 1 \frown 0^{\alpha} \frown 1}) = \sigma (l_{\nu \frown 0 \frown 0^{\alpha} \frown 1},r_{\nu \frown 0 \frown 0^{\alpha} \frown 1}) $$ 
for all $\alpha < \omega$.  This completes the construction of $(l_{\eta}, r_{\eta})_{\eta \in T_{n+1}}$, properties (1)--(3) are satisfied because of the inductive assumption.  We obtain $(l_{\eta}, r_{\eta})_{\eta \in 2^{<\omega}}$ as the union over all $n$ of $(l_{\eta},r_{\eta})_{\eta \in T_{n}}$.  

Now we check that with respect to the parameters $(l_{\eta})_{\eta \in 2^{<\omega}}$, $\varphi$ witnesses SOP$_{1}$.  Fix any path $\eta \in 2^{\omega}$, we have to check that $\{\varphi(x;l_{\eta | \alpha}) : \alpha < \omega\}$ is consistent.  But given any $n$, $l_{\unlhd (\eta | n)} \subset T_{n}$ and by (1), $l_{\unlhd (\eta | n)} \equiv_{C} (c_{\alpha, 0})_{\alpha \leq n}$ hence $\{\varphi(x;l_{\eta | \alpha}) : \alpha \leq n\}$ is consistent, as $\{\varphi(x;c_{\alpha,0}) : \alpha \leq n\}$ is consistent, by hypothesis.  Then $\{\varphi(x;l_{\eta | \alpha}) : \alpha < \omega\}$ is consistent by compactness.  

Now fix $\eta \perp \nu \in 2^{<\omega}$ so that $(\eta \wedge \nu) \frown 0 \unlhd \eta$ and $(\eta \wedge \nu) \frown 1 = \nu$.  We must check $\{\varphi(x;l_{\eta}), \psi(x;l_{\nu})\}$ is inconsistent.  As $\nu = (\eta \wedge \nu) \frown 1$, we know that $l_{\nu} = l_{(\eta \wedge \nu) \frown 1} = r_{(\eta \wedge \nu) \frown 0}$ by (2).  Let $\xi = (\eta \wedge \nu) \frown 0$. Then $\xi \unlhd \eta$ and $l_{\nu} = r_{\xi}$ so it suffices to show $\{\varphi(x;l_{\eta}), \varphi(x;r_{\xi})\}$ is inconsistent.  Let $n = l(\eta)$ and $m = l(\xi)$.  Then $m \leq n$ and by (1), we have $(l_{\eta}, r_{\xi}) \equiv_{C} (c_{n,0},c_{m,1})$.  By hypothesis, this implies $\{\varphi(x;l_{\eta}), \varphi(x;r_{\xi})\}$ is inconsistent, so we finish.  
\end{proof}

\begin{defn}
Suppose $\ind$ is an $\text{Aut}(\mathbb{M})$-invariant ternary relation on small subsets of $\mathbb{M}$.  
\begin{enumerate}
\item We say $\ind$ satisfies \emph{weak independent amalgamation over models} if, given $M \models T$, $b_{0}c_{0} \equiv_{M} b_{1}c_{1}$ satisfying $b_{i} \ind_{M} c_{i}$ for $i = 0,1$ and $c_{0} \ind_{M} c_{1}$, there is $b$ satisfying $bc_{0} \equiv_{M} bc_{1} \equiv_{M} b_{0}c_{0}$.  
\item We say $\ind$ satisfies \emph{independent amalgamation over models} if, given $M \models T$, $b_{0}\equiv_{M} b_{1}$ satisfying $b_{i} \ind_{M} c_{i}$ for $i = 0,1$ and $c_{0} \ind_{M} c_{1}$, there is $b$ satisfying $bc_{0} \equiv_{M} b_{0}c_{0}$ and $bc_{1} \equiv_{M} b_{1}c_{1}$.  
\item We say $\ind$ satisfies \emph{stationarity} \emph{over models} if: given $M \models T$, if $b_{0} \equiv_{M} b_{1}$ and $b_{0}\ind_{M}c,b_{1}\ind_{M}c$ then $b_{0}\equiv_{Mc}b_{1}$.
\end{enumerate}
\end{defn}

\begin{defn}
Suppose $A,B,C$ are small subsets of the monster $\mathbb{M}$.  
\begin{enumerate}
\item We say $A \ind^{i}_{C} B$ if and only if $\text{tp}(A/BC)$ can be extended to a global type Lascar-invariant over $C$.  We denote its dual by $\ind^{ci}$ - i.e. $A \ind_{C}^{i} B$ holds if and only if $B \ind_{C}^{ci} A$.  
\item We say $A \ind^{u}_{C} B$ if and only if $\text{tp}(A/BC)$ is finitely satisfiable in $C$.  We denote its dual by $\ind^{h}$ - i.e. $A \ind_{C}^{h} B$ if and only if $B \ind_{C}^{u} A$.  
\end{enumerate}
\end{defn}

Suppose $q(x)$ and $r(y)$ are global $M$-invariant types.  Recall that the product $q(x) \otimes r(y) \in S_{xy}(\mathbb{M})$ is defined by $q(x) \otimes r(y) = \text{tp}(ab/\mathbb{M})$ where $b \models r$ and $a \models q|_{\mathbb{M}b}$.  

\begin{prop}\label{weakindepamalg}
Fix a model \(M \models T\).  Suppose \(c_{1} \ind^{i}_{M} c_{0}\), \(c_{j} \ind^{i}_{M} b_{j}\) for \(j =0,1\) and \(b_{0}c_{0} \equiv_{M} b_{1}c_{1}\), but there is no \(b\) such that \(bc_{0} \equiv_{M} bc_{1} \equiv_{M} b_{0}c_{0}\).  Then $T$ has $\SOP_1$.    
\end{prop}

\begin{proof}
Let \(p(x;y) = \text{tp}(b_{0}c_{0}/M)\).  Our assumption entails that \(p(x;c_{0}) \cup p(x;c_{1})\) is inconsistent.  By compactness, there is some $\varphi(x;y) \in p(x;y)$ so that $\{\varphi(x;c_{0}), \varphi(x;c_{1})\}$ is inconsistent.  Fix a global $M$-invariant type $r$ so that $c_{0} \models r|_{M_{b_{0}}}$ and a global $M$-invariant type $q$ so that $c_{1} \models q|_{M_{c_{0}}}$.  Then $c_{1}c_{0} \models (q \otimes r)|_{M}$.  Let $(c_{1}^{i},c_{0}^{i})_{1 \leq i < \omega}$ be a Morley sequence in $(q \otimes r)|_{Mb_{0}c_{0}c_{1}}$ and put $(c^{0}_{1},c^{0}_{0}) = (c_{1},c_{0})$.   

First, we note that $b_{0} \models \{\varphi(x;c_{0}^{i}) : i < \omega\}$ so \emph{a fortiori} $\{\varphi(x;c_{0}^{i}) :  i< \omega\}$ is consistent.  Secondly, for any $N < \omega$, we have
$$
(c^{1}_{0}c^{1}_{1})\ldots (c^{N}_{0}c^{N}_{1}) \ind^{i}_{M} c_{0}c_{1}
$$ 
so by $M$-invariance and the fact that $c_{0} \equiv_{M} c_{1}$, we know that
$$
c_{0} \equiv_{Mc_{0}^{1}c_{1}^{1}\ldots c_{0}^{N}c_{1}^{N}} c_{1}
$$
Next, as \(c^{1}_{1} \models q|_{Mc_{0}c_{1}}\), we have \(c^{1}_{1} \equiv_{Mc_{0}} c_{1}\) and therefore $\{\varphi(x;c_{0}),\varphi(x;c_{1}^{1})\}$ is inconsistent.  
As \((c_{1}^{i},c_{0}^{i})_{i < \omega}\) is an \(M\)-indiscernible sequence, we've shown the following.
\begin{enumerate}
\item If \(X \subseteq \omega\) and $j < k$ for all $k \in X$, then $\{\varphi(x;c_{0}^{k}) : k \in X\} \cup \{\varphi(x;c_{i}^{j})\}$ is consistent for $i = 0,1$.  
\item If \(X \subseteq \omega\) and $j < k$ for all $k \in X$, then, writing \(\overline{c}_{X}\) for an enumeration of \(\{c_{0}^{k}c_{1}^{k} : k \in X\}\), we have \(c^{j}_{0} \equiv_{M \overline{c}_{X}} c^{j}_{1}\).  
\item If $j \leq k$ then $\{\varphi(x;c_{0}^{j}), \varphi(x;c_{1}^{k})\}$ is inconsistent.  
\end{enumerate}
Now by compactness (reversing the ordering on the sequence of pairs), we can find an array \((d_{i,j})_{i < \omega, j < 2}\) such that  the following holds.
\begin{enumerate}
\item For all \(i < \omega\), \(d_{i,0} \equiv_{M \overline{d}_{<i}} d_{i,1}\);
\item \(\{\varphi(x;d_{i,0}) : i < \omega\}\) is consistent; 
\item \(j \leq i \implies \{\varphi(x;d_{i,0}), \varphi(x;d_{j,1})\}\) is inconsistent.
\end{enumerate}
By Lemma \ref{automorphisms}, this implies $T$ has $\SOP_1$.  
\end{proof}

The following argument is an elaboration on \cite[Proposition 6.20]{MR3129735}, which, in turn, was an elaboration on an argument of Kim \cite[Proposition 2.6]{KimInThere}.

\begin{prop}\label{amalgfail}
Assume \(\varphi(x;y)\) witnesses $\SOP_1$.  Then there are \(M\), \(c_{0}, c_{1}, b_{0},b_{1}\) so that \(c_{0} \ind^{u}_{M} c_{1}\), \(c_{0} \ind_{M}^{u} b_{0}\), \(c_{1} \ind_{M}^{u} b_{1}\), \(b_{0}c_{0} \equiv_{M} b_{1}c_{1}\) and \(\models \varphi(b_{0},c_{0}) \wedge \varphi(b_{1}, c_{1})\) but \(\varphi(x;c_{0}) \wedge \varphi(x;c_{1})\) is inconsistent.  
\end{prop}

\begin{proof}
 Suppose \(T\) has $\SOP_1$ witnessed by \(\varphi\).  By compactness, we may assume that we have a tree of tuples \((a_{\eta})_{\eta \in 2^{< \kappa}}\) for \(\kappa\) large enough ($\geq 2^{|T|}$ suffices) so that
\begin{itemize}
 \item For all \(\eta \in 2^{\kappa}\), \(\{\varphi(x;a_{\eta | \alpha}) : \alpha < \kappa\}\) is consistent 
\item \(\eta \frown 0 \underline{\vartriangleleft} \nu \in 2^{<\kappa}\), then \(\{\varphi(x;a_{\eta \frown 1}), \varphi(x;a_{\nu})\}\) is inconsistent.
\end{itemize}
Fix a Skolemization \(T^{\Sk}\) of \(T\) and in what follows, we'll work modulo this expanded theory.  We will construct a sequence \((\eta_{i}, \nu_{i})_{i < \omega}\) of elements of \(2^{<\kappa}\) satisfying the following.
\begin{enumerate}
 \item \(a_{\nu_{i}}, a_{\eta_{i}}\) have the same type over \(a_{\eta_{<i}}, a_{\nu_{<i}}\)
\item If \(i < j\) then \(\eta_{i} \vartriangleleft \eta_{j}\) and \(\eta_{i} \vartriangleleft \nu_{j}\).  
\item \( (\eta_{i} \wedge \nu_{i}) \frown 0 \underline{\vartriangleleft} \eta_{i}\) and \((\eta_{i} \wedge \nu_{i}) \frown 1 = \nu_{i}\).
\end{enumerate}
Given \(n\), suppose \((\eta_{i}, \nu_{i} : i < n)\) have been chosen satisfying (1)-(3).  Consider the sequence \((a_{\eta_{n-1} \frown 0^{\alpha} \frown 1} : \alpha < \kappa)\).  As \(\kappa\) is large enough, there are \(\alpha < \beta < \kappa\) so that \(a_{\eta_{n-1}\frown 0^{\alpha} \frown 1}, a_{\eta_{n-1} \frown 0^{\beta} \frown 1}\) have the same type over \((a_{\eta_{<n}}, a_{\nu_{<n}})\).  Let \(\nu_{n} = \eta_{n-1}\frown 0^{\alpha} \frown 1\) and \(\eta_{n} = \eta_{n-1} \frown 0^{\beta} \frown 1\).  Now (1) and (2) are clearly satisfied, and, as \(\alpha < \beta\), \((\eta_{n} \wedge \nu_{n}) = \eta_{n-1}\frown 0^{\alpha}\) so (3) follows.  This completes the construction.

Now we claim that \((a_{\eta_{i}}, a_{\nu_{i}})_{i < \omega}\) satisfies:
\begin{enumerate}
\setcounter{enumi}{3}
 \item \(\{\varphi(x;a_{\eta_{i}}) : i < \omega\}\) is consistent,
\item \(a_{\nu_{i}},a_{\eta_{i}}\) have the same type over \(a_{\nu_{<i}},a_{\eta_{<i}}\),
\item \(\{\varphi(x;a_{\nu_{i}}), \varphi(x;a_{\nu_{j}})\}\) is inconsistent for \(i \neq j\).
\end{enumerate}
Here (5) is immediate from our choice of the sequence and we get (4) since \(i < j\) implies \(\eta_{i} \vartriangleleft \eta_{j}\) and paths are consistent.  To see (6), notice that if \(i < j\) then as \(\eta_{i} \vartriangleleft \nu_{j}\) and \(\eta_{i} \perp \nu_{i}\), we have \((\nu_{i} \wedge \nu_{j}) = (\eta_{i} \wedge \nu_{i})\) and hence \((\nu_{i} \wedge \nu_{j}) \frown 0 \trianglelefteq \nu_{j}\) and \(\nu_{i} = (\nu_{i} \wedge \nu_{j}) \frown 1\) from which (6) follows, using $\SOP_1$.  

By compactness and Ramsey,  we can find \(b\) and \((a_{\eta_{i}}, a_{\nu_{i}})_{i \leq \omega+1}\) indiscernible over \(b\), satisfying (4)-(6), and such that \(b \models \{\varphi(x;a_{\eta_{i}}) : i \leq \omega+1\}\).  Let \(M = \Sk(a_{\eta_{i}},a_{\nu_{i}})_{i < \omega}\).  Then we have \(a_{\eta_{\omega+1}} \ind^{u}_{M} b\) and \(a_{\nu_{\omega}} \ind_{M}^{u} a_{\eta_{\omega+1}}\) by indiscernibility.  As \(a_{\nu_{\omega}}, a_{\eta_{\omega}}\) start an \(M\)-indiscernible sequence, there is \(\sigma \in \text{Aut}(\mathbb{M}/M)\) sending \(a_{\eta_{\omega}} \mapsto a_{\nu_{\omega}}\).  Let \(b' = \sigma(b)\).  Then \(b' \equiv_{M} b\), \(a_{\nu_{\omega}} \ind_{M}^{u} b'\) (as \(a_{\eta_{\omega}} \ind_{M}^{u} b\) by indiscernibility) and \(\models \varphi(b'; a_{\nu_{\omega}})\).  But \(\{\varphi(x;a_{\eta_{\omega+1}}), \varphi(x;a_{\nu_{\omega}})\}\) is inconsistent by (5) and (6).  As \(\varphi\) is an \(L\)-formula, \(M\) is, in particular, an \(L\)-model and \(\ind^{u}\) in the sense of \(T^{\Sk}\) implies \(\ind^{u}\) in the sense of \(T\).  
\end{proof} 

\begin{thm}
The following are equivalent.
\begin{enumerate}
\item $\ind^{ci}$ satisfies weak independent amalgamation:  given any $M \models T$, $b_{0}c_{0} \equiv_{M} b_{1}c_{1}$ so that $c_{1} \ind^{i}_{M} c_{0}$ and $c_{j} \ind^{i}_{M} b_{j}$ for $j = 0,1$, there is $b$ so that $bc_{0} \equiv_{M} bc_{1} \equiv_{M} b_{0}c_{0}$.
\item $\ind^{h}$ satisfies weak independent amalgamation:  given any $M \models T$, $b_{0}c_{0} \equiv_{M} b_{1}c_{1}$ so that $c_{1} \ind^{u}_{M} c_{0}$ and $c_{j} \ind^{u}_{M} b_{j}$ for $j = 0,1$, there is $b$ so that $bc_{0} \equiv_{M} bc_{1} \equiv_{M} b_{0}c_{0}$.
\item $T$ is NSOP$_{1}$.  
\end{enumerate}
\end{thm}

\begin{proof}
(1)$\implies$(2) is clear.

(2)$\implies$(3) is Proposition \ref{amalgfail}.

(3)$\implies$(1) is Proposition \ref{weakindepamalg}.
\end{proof}

\begin{prop}\label{criterion}
Assume there is an \(\text{Aut}(\mathbb{M})\)-invariant independence relation \(\ind\) on small subsets of the monster \(\mathbb{M} \models T\) such that it satisfies the following properties, for an arbitrary \(M \models T\) and arbitrary tuples from $\mathbb{M}$.
\begin{enumerate}
\item Strong finite character:  if \(a \nind_{M} b\), then there is a formula \(\varphi(x,b,m) \in \text{tp}(a/bM)\) such that for any \(a' \models \varphi(x,b,m)\), \(a' \nind_{M} b\).
\item Existence over models:  \(M \models T\) implies \(a \ind_{M} M\) for any \(a\).
\item Monotonicity: \(aa' \ind_{M} bb'\) \(\implies\) \(a \ind_{M} b\).
\item Symmetry: \(a \ind_{M} b \iff b \ind_{M} a\).
\item Independent amalgamation: \(c_{0} \ind_{M} c_{1}\), \(b_{0} \ind_{M} c_{0}\), \(b_{1} \ind_{M} c_{1}\), \(b_{0} \equiv_{M} b_{1}\) implies there exists \(b\) with \(b \equiv_{c_{0}M} b_{0}\), \(b \equiv_{c_{1}M} b_{1}\).  
\end{enumerate}
Then \(T\) is NSOP\(_{1}\).  
\end{prop}

\proof{
\textbf{Claim}  Let \(M \models T\), then \(a \ind^{u}_{M} b \implies a \ind_{M} b\).  

\emph{Proof of claim.}  If \(a \nind_{M} b\) then by strong finite character, there is some \(\varphi(x;m,b) \in \text{tp}(a/Mb)\) so that \(a' \nind_{M} b\) for any \(a'\) with \(\models \varphi(a';m,b)\).  However, as \(a \ind_{M}^{u} b\), it follows that there is some \(a' \in M\) such that \(\models \varphi(a';m,b)\).  Then \(b \nind_{M} a'\) by symmetry and \(b \nind_{M} M\) by monotonicity, contradicting existence.  

Now assume towards contradiction that \(T\) has $\SOP_1$, and let \(M, c_{0}, c_{1}, b_{0},b_{1}, \varphi(x;y)\) as given in Proposition \ref{amalgfail}.  By the claim and symmetry of \(\ind\) we have \(c_{0} \ind_{M} c_{1}\), \(b_{0} \ind_{M} c_{0}, b_{1} \ind_{M} c_{1}\).  As \(\ind\) satisfies independent amalgamation over models, there is some \(b \ind_{M} c_{0}c_{1}\), \(b \equiv_{c_{0}M} b_{0}\), \(b \equiv_{c_{1}M} b_{1}\).  This contradicts the inconsistency of \(\{\varphi(x;c_{0}), \varphi(x;c_{1})\}\).  
}

\begin{rem}

\begin{enumerate}
\item We don't require the local character here, as it would then give simplicity
according to the theorem of Kim and Pillay \cite{KimPillay}.
\item We do require \emph{strong} finite character, which is not required
in Adler's definition of mock stability and mock simplicity (see \cite[the discussion after Definition 12]{AdlerMock}).
Indeed, there are mock stable examples arbitrarily high in the $\SOP_{n}$
hierarchy.
\end{enumerate}
\end{rem}

\section{Examples of $\NSOP_1$ theories}

\subsection{Vector spaces with a generic bilinear form}

Let $L$ denote the language with two sorts $V$ and $K$ containing the language of abelian groups for variables from $V$, the language of rings for variables from $K$, a function $\cdot : K \times V \to V$, and a function $[\text{ }]: V \times V \to K$.  $T_{\infty}$ is the model companion of the $L$-theory asserting that $K$ is a field, $V$ is a $K$-vector space of infinite dimension with the action of $K$ given by $\cdot$, and $[\text{ }]$ is a non-degenerate bilinear form on $V$.  If $(K,V) \models T_{\infty}$ then $K$ is an algebraically closed field.  

The theory $T_{\infty}$ was introduced by Nicolas Granger in \cite{Granger}, who observed that its completions are not simple, but nonetheless have a notion of independence called $\Gamma$-non-forking satisfying essentially all properties of forking in stable theories, except local character.  

\begin{defn}
We are using the notation from \cite[Notation 9.2.4]{Granger}. Let \(M = (V, \tilde{K})\) be a sufficiently saturated model of \(T_{\infty}\).  Let \(A \subseteq B \subset M\) and \(c \in M\) with \(c\) a singleton.  Let \(c \ind^{\Gamma}_{A} B\) be the assertion that \(K_{Ac} \ind_{K_{A}} K_{B}\) in the sense of non-forking independence for algebraically closed fields and one of the following holds:
\begin{enumerate}
\item \(c \in \tilde{K}\)
\item \(c \in \langle A \rangle\)
\item \(c \not\in \langle B \rangle\) and \([c,B]\) is \(\Phi\)-independent over \(A\), 
\end{enumerate}
where `\([c,B]\) is \(\Phi\)-independent over A' means that whenever \(\{b_{0}, \ldots, b_{n-1}\}\) is a linearly independent set in \(B_{V} \cap (V \setminus \langle A \rangle)\) then the set \(\{[c,b_{0}], \ldots, [c,b_{n-1}]\}\) is algebraically independent over the field \(K_{B}(K_{Ac})\).  

By induction, for \(c = (c_{0}, \ldots, c_{m})\) define \(c \ind_{A}^{\Gamma} B\) by 
\[
c \ind_{A}^{\Gamma} B \iff (c_{0}, \ldots, c_{m-1}) \ind^{\Gamma}_{A} B \text{ and } c_{m} \ind^{\Gamma}_{Ac_{0}\ldots c_{m-1}} B c_{0}\ldots c_{m-1}.
\]
\end{defn}

\begin{fact}\cite[Theorem 12.2.2]{Granger}
Let $M = (V,K) \models T_{\infty}$.  Then the relation on subsets of $M$ given by $\Gamma$-non-forking is automorphism invariant, symmetric, and transitive.  Moreover, it satisfies extension, finite character, and stationarity over a model.  
\end{fact}

\begin{lem}
If \(c\) is a tuple and \(A, B\) are small sets with \(c \nind^{\Gamma}_{A} B\), then there is a formula \(\varphi(x;a,b) \in \text{tp}(c/AB)\) so that
\[
\models \varphi(c';a,b) \implies c' \nind^{\Gamma}_{A} B.
\]
\end{lem}

\begin{proof}
Suppose \(c = (c_{0}, \ldots, c_{n-1})\) a tuple and \(c \nind^{\Gamma}_{A} B\).  Let \(k\) be maximal so that \((c_{0}, \ldots, c_{k-1}) \ind_{A}^{\Gamma} B\).  It follows that \(c_{k} \nind^{\Gamma}_{Ac_{0}\ldots c_{k-1}} Bc_{0}\ldots c_{k-1}\), so one of the following possibilities occurs:
\begin{enumerate}
\item \(K_{Ac_{0}\ldots c_{k}} \nind^{\ACF}_{K_{Ac_{0}\ldots c_{k-1}}} K_{Bc_{0}\ldots c_{k-1}}\)
\item \(c_{k} \in \langle Bc_{0}\ldots c_{k-1} \rangle \setminus \langle A c_{0} \ldots c_{k-1} \rangle\)
\item There is a linearly independent set \(\{d_{0}, \ldots, d_{l-1}\}\) from \((Bc_{0}\ldots c_{k-1})_{V} \cap (V \setminus \langle A c_{0} \ldots c_{k-1} \rangle)\) so that \(\{[c_{k},d_{0}], \ldots, [c_{k},d_{l-1}]\}\) is not algebraically independent over \(K_{Bc_{0}\ldots c_{k-1}}(K_{Ac_{0}\ldots c_{k}})\).  
\end{enumerate}
The existence of the desired formula requires an argument only in case (3).  In this case, there is a nonzero polynomial \(p(x_{0}, \ldots,x_{l-1};a,b,c_{0},\ldots,c_{k})\) with coefficients in \(K_{Bc_{0}\ldots c_{k-1}}(K_{Ac_{0}\ldots c_{k}})\) so that \(p([c_{k}, d_{0}], \ldots, [c_{k},d_{l-1}];a,b,c_{0},\ldots,c_{k}) = 0\).  By reindexing the \(d_{j}\), we may assume that there is \(m \leq l\) so that \(d_{j} = c_{i_{j}}\) for \(j < m\) and \(d_{j} \in B\) for \(j \geq m\).  Let \(d = (d_{m}, \ldots, d_{l-1})\).  Writing \(y = (y_{0}, \ldots, y_{k})\), let \(\chi(y;a,b,d)\) be the formula which asserts the following:
\begin{enumerate}
\item the polynomial \(p(x_{0},\ldots,x_{l-1};a,b,y)\) is a nonzero polynomial;
\item the set $\{y_{i_{0}}, \ldots, y_{i_{m-1}}\} \cup \{d_{m},\ldots, d_{l-1}\}$ is linearly independent;
\item \(p([y_{k},y_{i_{0}}],\ldots, [y_{k}, y_{i_{m-1}}],[y_{k},d_{m}],\ldots,[y_{k},d_{l-1}];a,b,y) = 0\)
\end{enumerate} 
Then \(\chi(y;a,b,d) \in \text{tp}(c/B)\) and if \(\models \chi(c';a,b,d)\) then it is easy to check \(c' \nind^{\Gamma}_{A} B\).  
\end{proof}

\begin{cor}
The two-sorted theory \(T_{\infty}\) of infinite dimensional vector spaces over algebraically closed fields with a generic bilinear form is NSOP\(_{1}\).  
\end{cor}

\subsection{$\omega$-free PAC fields of characteristic zero}

\begin{defn}
A field $F$ is called \emph{pseudo-algebraically closed} if every absolutely irreducible variety defined over $F$ has an $F$-rational point.  A field $F$ is called $\omega$\emph{-free} if it has a countable elementary substructure $F_{0}$ with $\mathcal{G}(F_{0}) \cong \hat{\mathbb{F}}_{\omega}$, the free profinite group on countably many generators.  
\end{defn}

In \cite{ZoePAC2}, Chatzidakis showed that a PAC field has a simple theory if and only if it has finitely many degree $n$ extensions for all $n$ so an $\omega$-free PAC field will not be simple.  Nonetheless, she showed that an $\omega$-free PAC field comes equipped with a notion of independence which is well-behaved.

\begin{fact}\cite{ZoeFree, ZoeOmegaFreeNotes}
Suppose $F$ is a sufficiently saturated $\omega$-free PAC field of characteristic zero.  Given $A = \text{acl}(A)$, $B = \text{acl}(B)$, $C = \text{acl}(C)$ with $C \subseteq A,B \subseteq F$, write $A \ind^{I}_{C} B$ to indicate that $A\ind^{\ACF}_C B$ and $A^{\text{alg}}B^{\text{alg}} \cap \text{acl}(AB) = AB$.  Extend this to non-algebraically closed sets by stipulating $a \ind^{I}_{D} b$ holds if and only if $\text{acl}(aD) \ind^{I}_{\text{acl}(D)} \text{acl}(bD)$.  Then $\ind^{I}$ satisfies existence over models, monotonicity, symmetry, and independent amalgamation over models.  
\end{fact}

It remains to check that $\ind^{I}$ satisfies strong finite character. The proof of it was pointed out to us by Zo\'e Chatzidakis, whom we would like to thank.

\begin{lem} \label{lem: strong finite char PAC}
Suppose $F$ is a sufficiently saturated $\omega$-free PAC field of characteristic zero.  If $a,b,c$ are tuples from $F$ and $a \ind^{I}_{c} b$ then there is a formula $\varphi(x;b,c) \in \text{tp}(a/bc)$ so that if $F \models \varphi(a';b,c)$ then $a' \nind^{I}_{c} b$.  
\end{lem}

\begin{proof}
If $a \nind^{\ACF}_{c} b$, then the existence of such a formula is clear, so we may assume $a \ind_{c}^{\ACF} b$.  As $a \nind^{I}_{c} b$, there are $\beta \in \langle cb \rangle^{\text{alg}}$, $\alpha \in \langle c a \rangle^{\text{alg}}$ not in $F$ such that $F(\alpha) = F(\beta)$ and $\beta \notin F\langle c \rangle^{\text{alg}}$.  We choose them so that $F(\beta)$ is Galois over $F$ (always possible since $F\cap
\langle ca \rangle^{\alg} \langle cb \rangle^{\alg}$ is Galois over $(F\cap \langle ca \rangle ^{\alg})(F\cap \langle cb \rangle^{\alg})
= \acl(ca)\acl(cb)$).  

Some of the conjugates of $\beta$ over $\langle cb \rangle$ might lie in $F \langle c\rangle^{\alg}$ and this will be witnessed by elements of $\text{acl}(cb) = F \cap \langle cb \rangle^{\text{alg}}$.  We choose an element $b'$ of $\acl(cb)$ such that $\langle cbb' \rangle$ contains $\langle c b \beta \rangle \cap F$ and $\langle c b  b' \rangle$ is closed under $\text{Aut}(\text{acl}(cb) / \langle cb \rangle)$.  Let the formula $\theta(y;b,c)$ isolate $\text{tp}(b'/bc)$.  

Let $P(Y,b,c)$ be a minimal polynomial of $b'$ over $\langle bc \rangle$, and let $Q(Z,Y,b,c)$ be such that $Q(Z,b',b,c)$ is a minimal polynomial of $\beta$ over $\langle c bb' \rangle$.  

~

\textbf{Claim.}  If $\models \theta(b_{1},b,c)$, then $P(b_{1},b,c) = 0$, $Q(Z,b_{1},b,c)$ is irreducible of degree $[\langle c b \beta \rangle : \langle cbb' \rangle ]$ and a solution of $Q$ defines a Galois extension, which is not contained in $F\langle c \rangle^{\text{alg}}$.

The first two assertions of the claim are immediate.  For the last one, assume that  $(b_1,b_2)$
satisfies $P(b_1,b,c)=0 \land Q(b_2,b_1,b,c)=0$, and that $Q(Z,b_1,b,c)$ is
irreducible and defines a Galois extension of the right degree (all this
is expressible in $\tp_F(b'/bc)$), but that $b_2\in F \langle c \rangle^{\alg}$. Then there is a
formula in $\tp_F(b_1/cb)$ which will say that such a $b_2$ exists, and is therefore not in
$\tp_F(b'/bc)$.

~

Similarly let $a' \in \acl(ac)$ be such that $\langle c a \alpha \rangle \cap F = \langle c aa'  \rangle$ and let $R(W,T,c)$ be such that $R(W,a,c)$ is a minimal polynomial of $a'$ over $\langle ca \rangle$ and let $S(X,W,T,c)$ be such that $S(X,a',a,c)$ is a minimal polynomial of $\alpha$ over $\langle c a a' \rangle$.  

The formula $\varphi(t,b,c)$ is a conjunction of the following assertions:
\begin{itemize}
\item $\exists y \theta(y,b,c)$,
\item $R(W,t,c)$ is not the trivial polynomial,
\item $(\exists w) R(w,t,c) = 0$ and $S(X,w,t,c)$ is irreducible over $F$ of degree $[\langle c a \alpha \rangle : \langle caa' \rangle ]$,
\item $(\forall z)[ Q(z,y,b,c) = 0 \to \text{``}F(z) \text{ contains a root of }S(X,w,t,c) = 0\text{''}$.
\end{itemize}
These statements are first-order using standard facts on interpretability of
finite algebraic extensions of a field in a field and definability of irreducibility (see e.g. \cite{van1984bounds}).

Assume now that $d$ satisfies $\varphi(t,b,c)$.  Let $y = b_{1}$ and $w = d_{1} \in F$ be as guaranteed to exist by $\varphi$, and let $b_{2}$ be a root of $Q(Z,b_{1},b,c) = 0$; then $F(b_{2})$ is a proper Galois extension of $F$ of degree $[\langle  cb \beta \rangle : \langle c b b'  \rangle ]$ which is not contained in $F \langle c  \rangle^{\text{alg}}$.  

Because $d$ satisfies $\varphi$, if $d_{2}$ satisfies $S(X,d_{1},d,c) = 0$, then $F(d_{2}) = F(b_{2})$.  As $F(b_{2}) \not\subseteq F \langle c \rangle^{\text{alg}}$, we necessarily have $d \not\in \langle c \rangle^{\text{alg}}$ and, therefore, either $d \nind^{\text{ACF}}_{c} b$ or, otherwise, $\langle cd \rangle^{\text{alg}} \langle cb \rangle^{\text{alg}} \cap F \neq \text{acl}(cd) \text{acl}(cb)$.  This shows $d \nind^{I}_{c} b$.  
\end{proof}

\begin{cor}
The theory of \(\omega\)-free PAC fields of characteristic 0 is $\NSOP_1$.  
\end{cor}

\subsection{Examples via Parametrization} \label{sec: parametrization}

In this subsection, we show how to construct NSOP$_1$ theories from simple ones.  We start with a simple theory $T$ obtained as the theory of a Fra\"iss\'e limit satisfying the strong amalgamation property and, by analogy with the theory of parametrized equivalence relations $T^{*}_{\text{feq}}$, form the parametrization of this structure.  We show that the resulting theories are NSOP$_1$ by proving an independence theorem for a natural independence notion associated to these theories.  The construction we perform here was studied by Baudisch \cite{Baudisch} in the context of arbitrary model complete theories eliminating $\exists^{\infty}$. We expect that our results hold in this greater generality as well, but our setting already encompasses many interesting examples and simplifies the study of amalgamation.  

We begin by recalling some facts from Fra\"iss\'e theory.  

\begin{defn}(SAP)
Suppose \(\mathbb{K}\) is a class of finite structures.  We say \(\mathbb{K}\) has the Strong Amalgamation Property (SAP) if given \(A,B,C \in \mathbb{K}\) and embeddings \(e: A \to B\) and \(f: A \to C\) there is a \(D \in \mathbb{K}\) and embeddings \(g: B \to D\) and \(h: C \to D\) so that the following diagram commutes:
\[
\xymatrix{
 & B \ar[rd]^{g}  &  \\
A \ar[ru]^{e} \ar[rd]^{f}
 & & D \\
& C \ar[ru]^{h} & }
\]
and, moreover, \((\text{im}g) \cap (\text{im}h) = \text{im} ge\) (and hence \(=\text{im}hf\), as well).  
\end{defn}

The following is a useful criterion for SAP:

\begin{fact}\cite{Hodg}
Suppose \(\mathbb{K}\) is the age of a countable ultrahomogeneous structure \(M\).  Then the following are equivalent:
\begin{enumerate}
\item \(\mathbb{K}\) has the strong amalgamation property.
\item \(M\) has no algebraicity.
\end{enumerate}
\end{fact}

Let \(\mathbb{K}\) denote a Fra\"iss\'e class in a finite relational language \(L = \langle R_{i} : i < k\rangle\) where each relation symbol \(R_{i}\) has arity \(n_{i}\).  Let $T$ the complete $L$-theory of the Fra\"iss\'e limit of $\mathbb{K}$.  We'll define a new language \(L_{\pfc}\) which contains two sorts \(P\) and \(O\).  For each \(i < k\), there is an \((n_{i}+1)\)-ary relation symbol \(R^{i}_{x}\) where \(x\) is a variable of sort \(P\) and the suppressed \(n_{i}\) variables belong to the sort \(O\).  

Given an \(L_{\pfc}\)-structure \(M\), it is convenient to write \(M = (A,B)\) where \(O(M) = A\) and \(P(M) = B\).  We will refer to elements named by $O$ as \emph{objects} and elements named by $P$ as \emph{parameters}.  Given \(b \in B\), we define the \emph{L-structure associated to b in M}, denoted \(A_{b}\), to be the \(L\)-structure interpreted in \(M\) with domain \(A\) and each relation symbol \(R_{i}\) interpreted by \(R^{i}_{b}(A)\).  If $b \in B$ and $C \subseteq A$, write $\langle C \rangle_{b}$ to denote the $L$-substructure of $A_{b}$ generated by $C$ (as we assume the language is relational, this will have $C$ as its domain).  

We describe a class of finite structures \(\mathbb{K}_{\pfc}\) to be the class defined in the following way.  Let 
\[
\mathbb{K}_{\pfc} = \left\{M = (A,B) \in \text{Mod}(L_{\pfc}) : |M| < \aleph_{0}, (\forall b \in B)(\exists D \in \mathbb{K}) \left(A_{b} \cong D\right) \right\}
\]

From now on, we'll assume \(\mathbb{K}\) also satisfies SAP.  

\begin{lem}\label{SAP}
\(\mathbb{K}_{\pfc}\) is a Fra\"iss\'e class satisfying the Strong Amalgamation Property (SAP).  
\end{lem}

\begin{proof}
HP is clear and, as we allow the empty structure to be a model in \(\mathbb{K}_{\pfc}\), JEP follows from SAP.  So we show SAP.  

First, we may assume that 3 models in the amalgamation diagram have the same set of parameters.  Suppose \((A,D)\), \((B,E)\) and \((C,F)\) are in \(\mathbb{K}_{\pfc}\) and we have embeddings
\[
\xymatrix{
 & (C,F)   \\
(A,D) \ar[ru]^{i} \ar[rd]^{j}
 &  \\
&(B,E)  }
\]
By moving \(F\) and \(E\) over \(D\) if necessary, we may assume that \(i\) and \(j\) are just the inclusion maps on parameters and that \(F \cap E = D\).  By SAP in \(\mathbb{K}\), for each \(d \in D\), there are embeddings \(f_{d},g_{d}\) and \(G_{d} \in \mathbb{K}\) so that the following diagram commutes,
\[
\xymatrix{
 & C_{d} \ar[rd]^{f_{d}}  &  \\
A_{d} \ar[ru]^{i} \ar[rd]^{j}
 & & G_{d} \\
& B_{d} \ar[ru]^{g_{d}} & }
\]
where \(i\) and \(j\) are the induced maps, so that \(f_{d}(C_{d}) \cap g_{d}(B_{b}) = (f_{d} \circ i)(A_{d})\).  Since the language is relational, HP implies that we may take \(G_{d} = f_{d}(C_{d}) \cup g_{d}(D_{d})\).  Moreover, we may choose \(f_{d}\) and \(g_{d}\) so that they are the same functions for all \(d \in D\) on the underlying sets \(C\) and \(B\) respectively.  Call these functions \(f\) and \(g\).  Let \(G\) be the underlying set of \(G_{d}\) for some (all) \(d \in B\).  Now define a structure \((G,E \cup F)\) so that for all \(d \in D = E \cap F\), \(G_{d}\) is as above, if \(a \in E \setminus F\), \(G_{a}\) is some structure in \(\mathbb{K}\) extending \(g(B_{a})\) and, likewise, if \(a \in F \setminus E\), \(G_{a}\) is some structure extending \(f(C_{a})\).  The functions \(f\) and \(g\) extend to embeddings \(f: (C,F) \to (G,E \cup F)\) and \(g: (B,E) \to (G,E \cup F)\) so that \(f\) and \(g\) are both inclusions on parameters.  By construction, it is clear that \(fi = gj\).  Moreover, \(fi(A) = f(C) \cap g(B)\) and \(fi(D) = f(E) \cap g(F)\), which establishes SAP in \(\mathbb{K}_{\pfc}\).  
\end{proof}

As $\mathbb{K}_{\pfc}$ is a Fra\"iss\'e class, there is a unique countable ultrahomogeneous $L_{\pfc}$-structure with age $\mathbb{K}_{\pfc}$.  Let $T_{\pfc}$ denote its theory.  By Fra\"iss\'e theory, this theory eliminates quantifiers and is $\aleph_{0}$-categorical.  

\begin{lem}\label{model}
Suppose \((A,B) \models T_{\pfc}\).  Then, for all \(b \in B\), \(A_{b} \models T\).  
\end{lem}

\begin{proof}
Since the property that for all \(b \in B\), \(A_{b} \models T\) is an elementary property, it suffices to check this when \((A,B)\) is the unique countable model of \(T_{\pfc}\).  If \(d,e \in A_{b}\) satisfy \(\text{tp}_{L}(d) = \text{tp}_{L}(e)\) then, by quantifier-elimination, it is easy to check \(\text{tp}_{L_{\pfc}}(b,d) = \text{tp}_{L_{\pfc}}(b,e)\) and ultrahomogeneity of \((A,B)\) implies there is an \(L_{\pfc}\)-automorphism of \((A,B)\) fixing \(b\) and taking \(d\) to \(e\).  The induced \(L\)-automorphism of \(A_{b}\) witnesses that \(A_{b}\) is ultrahomogeneous.  By Fra\"iss\'e theory there is up to isomorphism a unique countable ultrahomogeneous \(L\)-structure with age \(\mathbb{K}\) so \(A_{b}\) is isomorphic to a model of \(T\), so \(A_{b} \models T\).  
\end{proof}

Suppose \(\mathbb{M} = (\mathbb{A}, \mathbb{B})\) is a monster model of \(T_{\pfc}\).  Given a formula \(\varphi \in L\) and a parameter \(p \in \mathbb{B}\), define \(\varphi_{p} \in L_{\pfc}\) to be the formula obtained by replacing each occurrence of \(R_{i}\) by \(R^{i}_{p}\) and giving the objects their eponymous interpretations in \(\mathbb{A}_{p}\) -- formally, this defines \(\varphi_{p}\) for atomic \(\varphi\) and then the full definition follows by induction on the complexity of the formulas.  If \(C \subseteq \mathbb{A}\) is a set of objects and \(q\) is an \(L\)-type over \(C\) (considered as a subset of \(\mathbb{A}_{p}\)), we define the type \(q_{p}\) by 
\[
q_{p} = \{ \varphi_{p} : \varphi \in q\}.
\]
\begin{lem}\label{pasting1}
Suppose \(\{p_{i} : i < \alpha\} \subseteq \mathbb{B}\) is a collection of distinct parameters and \(q^{i} : i < \alpha)\) is a sequence of non-algebraic \(L\)-types over \(C \subseteq \mathbb{A}\) (possibly with repetition), where \(q^{i}\) is considered as a type in \(\mathbb{A}_{p_{i}}\).  Then the \(L_{\pfc}\)-type \(\bigcup_{i < \alpha} q^{i}_{p_{i}}\) is consistent.     
\end{lem}

\begin{proof}
By compactness, it suffices to consider the case where \(\alpha < \omega\) and when the \(q^{i}\) are all finite types.  Hence, we simply have to show 
\[
\mathbb{M} \models (\exists x) \bigwedge_{i < \alpha}  q^{i}_{p_{i}}(x).  
\]
Moreover, by quantifier-elimination in \(T\), we may assume that each \(q^{i}\) is quantifier-free.  For each \(i < \alpha\), let \(C_{i} \in \text{Age}(\mathbb{A}_{p_{i}})\) the finite substructure generated by the elements of \(C\) mentioned in all of the \(q^{i}\).  So, the underlying set of each \(C_{i}\) is the same, although the interpretations of the relations may differ.  Given any \(i < \alpha\), we know that 
\[
\mathbb{A}_{p_{i}} \models (\exists x) \bigwedge q^{i}_{p_{i}}(x)
\]
so there is \(D_{i} \in \text{Age}(\mathbb{A}_{p_{i}})\) containing a witness \(d_{i}\) to the above existential formula.  By non-algebraicity of each type, we may assume that \(d_{i} \not\in C_{i}\) and, by HP, that \(D_{i} = C \cup \{d_{i}\}\).  

Now define an \(L_{\pfc}\)-structure \(E\) with underlying set of objects \(C \cup \{*\}\) where \(*\) is some new element and its parameters are \(\{p_{i} : i < \alpha\}\), and the relations are interpreted so that for each \(i < \alpha\), the map is the identity on \(C\) and sends \(d_{i} \mapsto *\) is an isomorphism of \(L\)-structures from \(D_{i}\) to \(E_{p_{i}}\).  It is clear that \(E \in \mathbb{K}_{\pfc}\) so there is a copy \(F\) isomorphic over $C \cup \{p_{i} : i < \alpha\}$ to it in \(\text{Age}(\mathbb{M})\).  
Now 
\[
F \models (\exists x) \bigwedge_{i < \alpha}  q^{i}_{p_{i}}(x)
\]
and hence this is satisfied in \(\mathbb{M}\), so we're done.  
\end{proof}

\begin{lem}\label{pasting2}
Suppose \(A,B,C \subseteq \mathbb{A}\) are small sets of objects, \(F \subseteq \mathbb{B}\) is a small set of parameters, \(A \cap B \subseteq C\), and \(b_{0}, b_{1} \in \mathbb{B}\) satisfy \(b_{0} \equiv_{CF} b_{1}\).  Then there is some \(b \in \mathbb{B}\) so that \(b \equiv_{ACF} b_{0}\) and \(b \equiv_{BCF} b_{1}\) (all in \(L_{\pfc}\)).   
\end{lem}

\begin{proof}
Given a set \(D \subseteq \mathbb{A}\) and \(p \in \mathbb{B}\), recall that we write \(\langle D \rangle_{p}\) for the \(L\)-substructure of \(\mathbb{A}_{p}\) with underlying set \(D\).  By compactness, it suffices to prove the lemma when \(A,B,C,\) and \(F\) are finite.  By quantifier-elimination, demanding some \(b \in \mathbb{B}\) so that \(b \equiv_{AC} b_{0}\) and \(b \equiv_{BC} b_{1}\) is equivalent to asking that \(\langle AC \rangle_{b} \cong \langle AC \rangle_{b_{0}}\) and \(\langle BC \rangle_{b} \cong \langle AC \rangle_{b_{1}}\).  Now, as \(b_{0} \equiv_{C} b_{1}\), \(\langle C \rangle_{b_{0}}\) may be identified with \(\langle C \rangle_{b_{1}}\).  We may view \(C, \langle AC \rangle_{b_{0}}\), and \(\langle BC \rangle_{b_{1}}\) as elements of \(\mathbb{K}\).  In \(\mathbb{K}\), we have inclusions \(i: C \to \langle AC \rangle_{b_{0}}\) and \(j: C \to \langle BC \rangle_{b_{1}}\), so by SAP, there are embeddings \(f,g\) and a \(D\in \mathbb{K}\) so that the following diagram commutes
\[
\xymatrix{
 & \langle AC \rangle_{b_{0}}  \ar^{f}[rd] \\
C \ar^{i}[ru] \ar^{j}[rd]
 &  & D\\
& \langle BC \rangle_{b_{1}} \ar^{g}[ru]
 }
\]
where \(f(AC) \cap g(BC) = C\).  By HP, \(D\) may be taken to have \(f(AC) \cup g(BC)\) as its domain.  Since \(A \cap B \subseteq C\), \(D\) is isomorphic over \(C\) to an \(L\)-structure with underlying set \(A \cup B \cup C\), so we may assume that \(f\) and \(g\) are both inclusions.  Let \(b_{*}\) denote some new parameter element outside of \(F\) and define a structure with parameter set \(\{b_{*},b_{0},b_{1}\} \cup F\) and \(A \cup B \cup C\) as its set of objects so that \(\langle ABC \rangle_{b_{*}} \cong D\).  This clearly defines a structure in \(\mathbb{K}_{\pfc}\).  In the substructure with only \(A \cup C\) as the set of objects, there is an automorphism fixing \(F\) taking \(b_{*}\) to \(b_{0}\).  This shows that \(b_{*} \equiv_{ACF} b_{0}\) and a symmetric argument shows \(b_{*} \equiv_{BCF} b_{1}\).  It follows that we can find such a \(b_{*}\) in \(\mathbb{B}\).   
\end{proof}

Towards proving an independence theorem for \(T_{\pfc}\), we will define a notion of independence for parameterized structures.

\begin{defn}{(\(\ind^{\pfc}\))}
\text{ }
\begin{enumerate}
\item Suppose \(p\in \mathbb{B}\) is a parameter.  Suppose \(A,B,C \subseteq \mathbb{A}\).  We define \(\ind^{p}\) by
\[
A \ind^{p}_{C} B \text{ in } \mathbb{M} \iff A \ind_{C} B \text{ in } \mathbb{A}_{p},
\]
where the undecorated \(\ind\) on the right-hand side denotes the usual non-forking independence -- i.e. \(\text{tp}(A/BC)\) does not fork over \(C\).  
\item If \(A,B,C \subseteq \mathbb{A}\) and \(D,E,F \subseteq \mathbb{B}\), we define \(\ind^{\pfc}\) by
\[
A,D \ind^{\pfc}_{C,F} B,E \iff D \cap E \subseteq F, \text{ and for all } p \in F, A \ind_{C}^{p} B.
\]
\end{enumerate}
\end{defn}

\begin{prop}\label{indep}
Assume \(T\) is a simple theory.  Suppose \(A,B \subseteq \mathbb{A}\) are small sets of objects and \(D,E\subseteq \mathbb{B}\) are small sets of parameters and \(M = (C,F)\) is a small model of \(T_{\pfc}\) satisfying
\[
A,D \ind^{\pfc}_{C,F} B,E
\]
Suppose moreover that $a_{0}, a_{1}$ are tuples from  $\mathbb{A}$ and $b_{0}, b_{1}$ are tuples from $\mathbb{B}$ satisfying \(a_{0},b_{0} \ind^{\pfc}_{CF} A,D\), \(a_{1},b_{1} \ind^{\pfc}_{C,F} B,E\) and \(a_{0},b_{0} \equiv_{CF} a_{1},b_{1}\).  Then there are $a$ from $\mathbb{A}$ and $b$ from $\mathbb{B}$ so that \(a,b \equiv_{ACDF} a_{0},b_{0}\) and \(a,b \equiv_{BCEF} a_{1},b_{1}\).  
\end{prop}

\begin{proof}
First, we solve the amalgamation problem for objects.  Without loss of generality, \(D,E,F\) are pairwise disjoint.  By Lemma \ref{model}, we know that for each \(p \in F\), \(C_{p}\) is a model of \(T\).  By definition of \(\ind^{\pfc}\), we know that in \(\mathbb{A}_{p}\), we have \(A \ind^{p}_{C} B\), \(a_{0} \ind^{p}_{C} A\) and \(a_{1} \ind^{p}_{C} B\).  As \(T\) is simple, the independence theorem over a model implies that there is some tuple $a_{p}$ in $ \mathbb{A}_{p}$ such that \(a_{p} \equiv^{L}_{AC} a_{0}\), \(a_{p} \equiv^{L}_{BC} a_{1}\) and \(a_{p} \ind^{p}_{C} AB\).  For each \(p \in F\), let \(q^{p}(x) = \text{tp}_{L}(a_{p}/ABC)\) considered as an \(L\)-type in \(\mathbb{A}_{p}\).  By Lemma \ref{pasting1}, denoting the relativization of \(q^{p}\) to the parametrized language with respect to \(p\) by \(q^{p}_{p}\), we know that the type \(\bigcup_{p \in F} q^{p}_{p}\) is consistent.  Let \(a\) be a realization.  Then $a \equiv_{AC} a_{0}$ and $a \equiv_{BC} a_{1}$ in $\mathbb{A}_{p}$ for all $p \in F$ so $a \equiv_{ACF} a_{0}$ and $a \equiv_{BCF} a_{1}$.

Now we solve the problem for parameters. First assume that $b_0, b_1$ are \emph{singletons} in $\mathbb{B}$. Without loss of generality $b_0, b_1 \notin  F$ (as otherwise they are equal by assumption, and there is nothing to do). By quantifier-elimination, we need some \(b \not\in D \cup E \cup F\) so that \(\langle aAC \rangle_{b} \cong \langle a_{0} AC \rangle_{b_{0}}\) and \(\langle a BC \rangle_{b} \cong \langle a_{1}BC \rangle_{b_{1}}\).  First, find \(b_{2} \equiv_{ACF} b_{0}\) and \(b_{3} \equiv_{BCF} b_{1}\) outside of \(D\cup E \cup F\) so that \(\langle a AC \rangle_{b_{2}} \cong \langle a_{0}AC \rangle_{b_{0}}\) and \(\langle aBC \rangle_{b_{3}} \cong \langle a_{1}BC \rangle_{b_{1}}\).  So \(ab_{2} \equiv_{ACF} a_{0}b_{0}\) and \(ab_{3} \equiv_{BCF} a_{1}b_{1}\).  Now \(b_{2} \equiv_{aCF} b_{3}\) and \(aAC \cap aBC \subseteq aC\), so Lemma \ref{pasting2} applies and we can find a \(b\) so that \(\langle aAC \rangle_{b} \cong \langle aAC \rangle_{b_{2}}\) and \(\langle aBC \rangle_{b} \cong \langle aBC \rangle_{b_{3}}\), and we can take this \(b\) to be outside of \(D \cup E \cup F\).  Now as \(b\not \in D  \cup E \cup F\), we have \(ab \equiv_{ACDF} a_{0}b_{0}\) and \(ab \equiv_{BCEF} a_{1}b_{1}\).

Now let $b_0 = (b_{0,i}: i<k), b_1 = (b_{1,i}: i<k)$ be arbitrary tuples from $\mathbb{B}$. Without loss of generality, all of the elements in $\{b_{t,i} : i < k\}$ are pairwise-distinct, for $t \in \{ 0, 1 \}$. Let $S_t = \{ i < k : b_{t,i} \notin F\}$ for $t\in \{ 0, 1 \}$, note that $S_0 = S_1 = S$ as $b_0 \equiv_F b_1$. 
Repeatedly applying the argument above for singletons, we can find \emph{pairwise distinct}  $b'_i$ for $i \in S$ such that $a, b'_i \equiv_{ACDF} a_0, b_{0,i}$ and $a,b'_i \equiv_{BCEF} a_1, b_{1,i}$ for all $i \in S$. 
Let $b^* = (b^*_i : i <k)$ be defined by taking $b^*_i = b_{0,i} = b_{1,i}$ for all $i \notin S$ and $b^*_i = b'_i$ for all $i \in S$. As there are no relations in the language involving more than one element from the parameter sort except for the equality, it follows that \(a,b^* \equiv_{ACDF} a_{0},b_{0}\) and \(a,b^* \equiv_{BCEF} a_{1},b_{1}\) --- as wanted.

\end{proof}

\begin{thm}
Assume \(T\) is simple.  Then \(\ind^{\pfc}\) is an \(\text{Aut}(\mathbb{M})\)-invariant independence relation on small subsets of the monster \(\mathbb{M} \models T_{\pfc}\) such that it satisfies, for an arbitrary \(M \models T_{\pfc}\): 
\begin{enumerate}
\item strong finite character:  if \(a \nind^{\pfc}_{M} b\), then there is a formula \(\varphi(x,b,m) \in \text{tp}(a/bM)\) such that for any \(a' \models \varphi(x,b,m)\), \(a' \nind^{\pfc}_{M} b\);
\item existence over models:  \(M \models T_{\pfc}\) implies \(a \ind^{\pfc}_{M} M\) for any \(a\);
\item monotonicity: \(aa' \ind^{\pfc}_{M} bb'\) \(\implies\) \(a \ind^{\pfc}_{M} b\);
\item symmetry: \(a \ind^{\pfc}_{M} b \iff b \ind^{\pfc}_{M} a\);
\item independent amalgamation: \(c_{0} \ind^{\pfc}_{M} c_{1}\), \(b_{0} \ind^{\pfc}_{M} c_{0}\), \(b_{1} \ind^{\pfc}_{M} c_{1}\), \(b_{0} \equiv_{M} b_{1}\) implies there exists \(b\) with \(b \equiv_{c_{0}M} b_{0}\), \(b \equiv_{c_{1}M} b_{1}\).  
\end{enumerate}
\end{thm}

\begin{proof}
Automorphism invariance and (1)-(4) are immediate from the definition of \(\ind^{\pfc}\), using that \(T\) is simple and hence non-forking independence satisfies all these properties;  (5) was proven in Proposition \ref{indep}.
\end{proof}

\begin{cor}
Suppose \(T\) is a simple theory which is the theory of a Fra\"iss\'e limit of a Fra\"iss\'e class \(\mathbb{K}\) satisfying SAP.  Then \(T_{\pfc}\) is NSOP\(_{1}\).  Moreover, if the \(D\)-rank of \(T\) is \(\geq 2\), then \(T_{\pfc}\) is not simple.  
\end{cor}

\begin{proof}
By Proposition \ref{criterion}, \(T_{\pfc}\) is NSOP\(_{1}\), as \(\ind^{\pfc}\) gives an independence relation satisfying all the hypotheses.  So now we prove that \(T_{\text{pfc}}\) is not simple, under the assumption that the \(D\)-rank of \(T\) is \(\geq 2\).  This assumption implies that there is an \(L\)-formula \(\varphi(x;y)\) and an indiscernible sequence \((a_{i})_{i < \omega}\) so that \(\{\varphi(x;a_{i}) : i < \omega\}\) is \(k\)-inconsistent for some \(k\) and the set defined by \(\varphi(x;a_{i})\) is infinite.  Let \(M \models T\) be some model containing the sequence \((a_{i})_{i < \omega}\).  Construct an \(L_{\text{pfc}}\)-structure \(N\) with domain \(\omega \sqcup M\) and relations interpreted so that \(N \models R_{i}(b) \iff M \models R(b)\) for each tuple \(b \in M\), every \(i < \omega\), and relation symbol \(R\) of \(L\).  Extend \(N\) to \(\tilde{N} \models T_{\text{pfc}}\).  Let \(\psi(x;y,z)\) be the formula \(\varphi_{z}(x;y)\) and define an array \((b_{ij})_{i,j <\omega}\) by \(b_{ij} = (a_{j},i) \in  M \times \omega \subset \tilde{N}^{2}\).  Using Lemma \ref{pasting1}, it is easy to check that for all \(f: \omega \to \omega\), \(\bigcup_{i < \omega} \{\psi(x;b_{if(i)})\}\) is consistent.  Also \(\{\psi(x;b_{ij}) : j < \omega\}\) is \(k\)-inconsistent for all \(i\) so \(\psi\) witnesses $\TP_2$.  
\end{proof}

\begin{rem}For the above argument to work, we used that the formula witnessing dividing was non-algebraic --- this fails in many natural examples (e.g. the random graph).  However, given an \(L\)-structure \(M\), define the \emph{imaginary cover} of \(M\) as follows: let \(L'\) be the language \(L\) together with a new binary relation symbol \(E\) for an equivalence relation, and let \(\tilde{M}\) be the \(L'\)-structure obtained by replacing each element of \(M\) with an infinite \(E\)-class and defining the relations of \(L\) on \(\tilde{M}\) on the corresponding \(E\)-classes.  Now it is easy to check that \(\text{Age}(\tilde{M})\) has SAP, the theory of \(\tilde{M}\) is simple of \(D\)-rank at least 2. \end{rem} 

\begin{cor}
\(T^{*}_{feq}\) is NSOP\(_{1}\).
\end{cor}

\begin{proof}
The theory \(T\) of an equivalence relation with infinitely many infinite classes is a stable theory, obtained as the Fra\"iss\'e limit of all finite models of the theory of an equivalence relation.  This class has no algebraicity, so it satisfies SAP.  \(T_{\text{pfc}}\) is exactly \(T^{*}_{feq}\), so it is NSOP\(_{1}\).
\end{proof}

This result was claimed in \cite{ShUs:844}, but the proof is apparently incorrect due to an illegitimate use of tree-indiscernibles.  See the footnote on  \cite[p. 22]{Harr} for a discussion.  

\subsection{Theories approximated by simple theories}

In her thesis \cite{Harr}, Gwyneth Harrison-Shermoen considers theories that have a model approximated by a directed system \(\mathcal{H}\) of homogeneous substructures, each of which has a simple theory.  She proves that such theories carry an invariant independence notion \(\ind^{\lim}\) satisfying strong finite character, monotonicity, symmetry, and existence over a model (existence over a model is implied by Claim 3.3.4 in \cite{Harr}).  Finally, she observes that if non-forking independence \(\ind^{f}\) satisfies the independence theorem over algebraically closed sets for each model in \(\mathcal{H}\), then so does \(\ind^{lim}\) for the approximated theory.  Hence, we obtain the following:

\begin{cor}
Suppose \(T\) is a theory approximated, in the sense of Harrison-Shermoen, by a directed system of structures each with a simple theory in which \(\ind^{f}\) satisfies the independence theorem over algebraically closed sets.  Then $T$ is NSOP$_1$.  
\end{cor}
\section{Lemmas on preservation of indiscernibility}\label{sec: appendix}

\begin{lem}
Suppose \(\eta_{0}, \ldots, \eta_{l-1}\), \(\nu_{0}, \ldots, \nu_{l-1}\) are elements of \(\omega^{<\omega}\).  Let \(\overline{\eta}\) and \(\overline{\nu}\) denote enumerations of the \(\wedge\)-closures of \(\{\eta_{i} : i < l\}\) and \(\{\nu_{i} : i < l\}\) respectively.  Then if 
$$
\text{qftp}_{L_{s}}(\eta_{0}, \ldots, \eta_{l-1}) = \text{qftp}_{L_{s}}(\nu_{0}, \ldots, \nu_{l-1}),
$$
then 
\[
\text{qftp}_{L_{s}}(\overline{\eta}) = \text{qftp}_{L_{s}}(\overline{\nu}).  
\]
\end{lem}

\begin{proof}
 Easy.  See Remark 3.2 from \cite{KimKimScow}
\end{proof}

\begin{lem}\label{initseg}
 Let \(\eta_{0}, \ldots, \eta_{l-1}, \nu_{0}, \ldots, \nu_{l-1} \in \omega^{<\omega}\) be such that 
\[
 \text{qftp}_{L_{s}} (\eta_{0}, \ldots, \eta_{l-1}) = \text{qftp}_{L_{s}} (\nu_{0}, \ldots, \nu_{l-1}).
\]
Suppose \(i < l\) and \(\eta \vartriangleleft \eta_{i}\), \(\nu \vartriangleleft \nu_{i}\) with \(l(\eta) = l(\nu)\).  Then, setting \(\eta_{l} = \eta\) and \(\nu_{l} = \nu\), we have 
\[
 \text{qftp}_{L_{s}}(\eta_{0}, \ldots, \eta_{l}) = \text{qftp}_{L_{s}}(\nu_{0}, \ldots, \nu_{l}).  
\]

\end{lem}

\begin{proof}
 Without loss of generality, we may take \(\{\eta_{i} : i < l\}\) and \(\{\nu_{i} : i < l\}\) to be \(\wedge\)-closed, by the previous lemma.  Then  \(\{\eta_{i} : i < l+1\}\) and \(\{\nu_{i} : i < l+1\}\) are also \(\wedge\)-closed.  So we need only to check that for any \(j,j' < l+1\)
\begin{enumerate}
 \item \(\eta_{j}  \vartriangleleft \eta_{j'} \iff \nu_{j} \iff \nu_{j}'\)
\item \(\eta_{j} <_{lex} \eta_{j'} \iff \nu_{j} <_{lex} \nu_{j'}\)
\end{enumerate}
We have 3 cases.

\textbf{Case 1:}  \(j,j' < l\).

\((1)\) and \((2)\) follow by assumption.  

\textbf{Case 2:} \(j < l\) and \(j' =l\)
\begin{eqnarray*}
 \eta_{j} \vartriangleleft \eta_{l} &\iff & \eta_{j} \vartriangleleft \eta_{i}  \text{ and } l(\eta_{j}) \leq l(\eta_{l}) \\
&\iff& \eta_{j} \vartriangleleft \eta_{i} \wedge \bigvee_{k < l(\eta_{l})} P_{k}(\eta_{j}) \\
&\iff& \nu_{j} \vartriangleleft \nu_{i} \wedge \bigvee_{k < l(\nu_{l})} P_{k}(\nu_{j}) \\
&\iff& \nu_{j} \vartriangleleft \nu_{l}.\\
 \eta_{j} <_{lex} \eta_{l} &\iff& l(\eta_{j} \wedge \eta_{i}) < l(\eta_{l}) \text{ and } \eta_{j} <_{lex} \eta_{i} \\
&\iff& \left(\bigvee_{k < l(\eta_{l})} P_{k}(\eta_{j} \wedge \eta_{i}) \right) \wedge \eta_{j} <_{lex} \eta_{i} \\
&\iff& \left(\bigvee_{k < l(\nu_{l})} P_{k}(\nu_{j} \wedge \nu_{i}) \right) \wedge \nu_{j} <_{lex} \nu_{i} \\
&\iff& \nu_{l} <_{lex} \nu_{j}.
\end{eqnarray*}
\textbf{Case 3:} \(j = l\) and \(j'<l\)
\begin{eqnarray*}
 \eta_{l} \vartriangleleft \eta_{j} &\iff& \eta_{l} \vartriangleleft (\eta_{i} \wedge \eta_{j}) \\
&\iff& \bigvee_{l(\eta_{l}) < k \leq l(\eta_{i})} P_{k}((\eta_{i} \wedge \eta_{j})) \\
&\iff& \bigvee_{l(\nu_{l}) < k \leq l(\nu_{i})} P_{k}((\nu_{i} \wedge \nu_{j})) \\
&\iff& \nu_{l} \vartriangleleft \nu_{j} \\
 \eta_{l} <_{lex} \eta_{j} &\iff& \left( l(\eta_{j} \wedge \eta_{i}) < l(\eta_{l}) \right) \to \eta_{i} <_{lex} \eta_{j} \\
&\iff& \left( \bigvee_{k < l(\eta_{l})} P_{k}( \eta_{j} \wedge \eta_{i} ) \right) \to \eta_{i} <_{lex} \eta_{j} \\
&\iff& \left( \bigvee_{k < l(\nu_{l})} P_{k}( \nu_{j} \wedge \nu_{i} ) \right) \to \nu_{i} <_{lex} \nu_{j} \\
&\iff& \nu_{l} <_{lex} \nu_{j}.
\end{eqnarray*}

\end{proof}

\begin{lem}\label{widening}
 Let \((a_{\eta})_{\eta \in \omega^{<\omega}}\) be an \(s\)-indiscernible tree. If \((a'_{\eta})_{\eta \in \omega^{<\omega}}\) is the \(k\)-fold widening of \((a_{\eta})_{\eta \in \omega^{<\omega}}\) at level \(n\), then \((a'_{\eta})_{\eta \in \omega^{<\omega}}\) is also \(s\)-indiscernible.  
\end{lem}

\begin{proof}
 Pick \(\eta_{0}, \ldots, \eta_{l-1}\) and \(\nu_{0}, \ldots, \nu_{l-1}\) in \(\omega^{<\omega}\) so that 
\[
 \text{qftp}_{L_{s}} (\eta_{0}, \ldots, \eta_{l-1}) = \text{qftp}_{L_{s}}(\nu_{0}, \ldots, \nu_{l-1}).
\]
By Lemma \ref{initseg}, we may assume that \(\{\eta_{i} : i < l\}\) and \(\{\nu_{i} : i < l\}\) are both \(\wedge\)-closed and closed under initial segment.  Moreover, we may assume that these elements have been enumerated so that for some \(m \leq l\), \(l(\eta_{i}), l(\nu_{i}) < n\) if and only if \(i \geq m\).  So for each \(i < m\), we may write 
\begin{eqnarray*}
 \eta_{i} &=& \mu_{i} \frown \alpha_{i} \frown \xi_{i} \\
\nu_{i} &=& \upsilon_{i} \frown \beta_{i} \frown \rho_{i},
\end{eqnarray*}
where \(\mu_{i}, \upsilon_{i} \in \omega^{n-1}\), \(\alpha_{i}, \beta_{i} \in \omega\), and \(\xi_{i}, \rho_{i} \in \omega^{<\omega}\).  For each \(i < m\), let 
\begin{eqnarray*}
 \overline{\eta}_{i} &=& (\mu_{i} \frown (k\alpha_{i}) \frown \xi_{i}, \mu_{i} \frown (k \alpha_{i} + 1) \frown \xi_{i}, \ldots, \mu_{i} \frown (k\alpha_{i} + k - 1) \frown \xi_{i}) \\
\overline{\nu}_{i} &=& (\upsilon_{i} \frown (k \beta_{i}) \frown \rho_{i}, \upsilon_{i} \frown (k \beta_{i} + 1) \frown \rho_{i}, \ldots, \upsilon_{i} \frown (k \beta_{i} + k -1) \frown \rho_{i}).
\end{eqnarray*}
and for \(m \leq i < l\), let \(\overline{\eta}_{i} = \eta_{i}\), \(\overline{\nu}_{i} = \nu_{i}\).  Now we must show that 
\[
 \text{qftp}_{L_{s}}(\overline{\eta}_{0}, \ldots, \overline{\eta}_{l-1}) = \text{qftp}_{L_{s}}(\overline{\nu}_{0}, \ldots, \overline{\nu}_{l-1}).
\]
It is clear that the sets \(\bigcup_{i <l} \overline{\eta}_{i}\) and \(\bigcup_{i < l} \overline{\nu}_{i}\) are closed under initial segment.  They are also closed under \(\wedge\):  this is obvious for elements of length \(<n\)  and for elements of longer length whose meet has length \(< n\) by our assumptions.  On the other hand if, for some \(i,i' < l\) and \(j,j'<k\), \(l((\overline{\eta}_{i})_{j}), l((\overline{\nu}_{i'})_{j'}) \geq n\) and \(l((\overline{\eta}_{i})_{j} \wedge (\overline{\nu}_{i'})_{j'}) \geq n\), then if \(j = j'\), we have \((\overline{\eta}_{i})_{j} \wedge (\overline{\nu}_{i'})_{j'} = \overline{(\eta_{i} \wedge \eta_{i'})}_{j}\) and if \(j \neq j'\), then 
\((\overline{\eta}_{i})_{j} \wedge (\overline{\nu}_{i'})_{j'}\) is equal to the common initial segment of each element of length \(n-1\).  In the first case, the meet is enumerated in one of the tuples because our initial set of tuples was \(\wedge\)-closed, in the second case because it was taken to be closed under initial segment.  To check equality of the quantifier-free types, we have 3 cases:

\textbf{Case 1:} \(i, i' \geq m\)
Follows by assumption, as for any \(i \geq m\), \(\overline{\eta}_{i} = \eta_{i}\) and \(\overline{\nu}_{i} = \nu_{i}\).  

\textbf{Case 2:} \(i \geq m\), \(i' < m\) and \(j < k\) 
\begin{eqnarray*}
 \overline{\eta}_{i} \vartriangleleft (\overline{\eta}_{i'})_{j} &\iff& \overline{\nu}_{i} \vartriangleleft (\overline{\nu}_{i'})_{j} \\
\overline{\eta}_{i} <_{lex} (\overline{\eta}_{i'})_{j} &\iff& \overline{\nu}_{i} <_{lex} (\overline{\eta}_{i'})_{j} \\
(\overline{\eta}_{i'})_{j} <_{lex} \overline{\eta}_{i} &\iff& (\overline{\nu}_{i'})_{j} <_{lex}  \overline{\nu}_{i}
\end{eqnarray*}
\textbf{Case 3:} \(i,i'<m\) and \(j,j'< k\) 
\begin{eqnarray*}
 (\overline{\eta}_{i})_{j} \vartriangleleft (\overline{\eta}_{i'})_{j'} &\iff& \eta_{i} \vartriangleleft \eta_{i'} \text{ and } j = j' \\
&\iff& \nu_{i} \vartriangleleft \nu_{i'} \text{ and } j = j' \\
&\iff& (\overline{\nu}_{i})_{j} \vartriangleleft (\overline{\nu}_{i'})_{j'} \\
 (\overline{\eta}_{i})_{j} <_{lex} (\overline{\eta}_{i'})_{j'} &\iff& \left( \eta_{i}<_{lex} \eta_{j} \text{ and } (l(\eta_{i} \wedge \eta_{j})<n \text{ or } j = j') \right) \text{ or } \\ & & \left( l(\eta_{i} \wedge \eta_{i'}) \geq n  \text{ and } j < j' \right) \\
&\iff& \left( \nu_{i}<_{lex} \nu_{j} \text{ and } (l(\nu_{i} \wedge \nu_{j})<n \text{ or } j = j') \right) \text{ or } \\ & & \left( l(\nu_{i} \wedge \nu_{i'}) \geq n  \text{ and } j < j' \right) \\
&\iff& (\overline{\nu}_{i})_{j} <_{lex} (\overline{\nu}_{i'})_{j'}.
\end{eqnarray*}
\end{proof}

\begin{lem}\label{stretch}
 Let \((a_{\eta})_{\eta \in \omega^{<\omega}}\) be an \(s\)-indiscernible tree. If \((a''_{\eta})_{\eta \in \omega^{<\omega}}\) is the \(k\)-fold stretch of \((a_{\eta})_{\eta \in \omega^{<\omega}}\) at level \(n\), then \((a''_{\eta})_{\eta \in \omega^{<\omega}}\) is also \(s\)-indiscernible.  
\end{lem}

\begin{proof}
Given \(\eta \in \omega^{<\omega}\), let 
\[
\overline{\eta} = \left\{
\begin{matrix}
\eta & \text{ if } l(\eta) < n \\
(\eta,\eta \frown 0, \ldots , \eta \frown 0^{k-1}) & \text{ if } l(\eta) = n \\
\nu \frown 0^{k-1} \frown \xi & \text{ if } \eta = \nu \frown \xi, \text{ with } \nu \in \omega^{n}, \xi \neq \emptyset
\end{matrix}
\right.
\]
Pick \(\eta_{0}, \ldots, \eta_{l-1}, \nu_{0}, \ldots, \nu_{l-1} \in \omega^{<\omega}\) so that 
\[
\text{qftp}_{L_{s}}(\eta_{0}, \ldots, \eta_{l-1}) = \text{qftp}_{L_{s}}(\nu_{0}, \ldots, \nu_{l-1}),
\]
and, without loss of generality, we may suppose \(\{\eta_{i}: i < l\}\) and \(\{\nu_{i} : i < l\}\) are both \(\wedge\)-closed.  We must show that 
\[
\text{qftp}_{L_{s}}(\overline{\eta}_{0}, \ldots, \overline{\eta}_{l-1}) = \text{qftp}_{L_{s}}(\overline{\nu}_{0}, \ldots, \overline{\nu}_{l-1}).
\]
Assume that $\{\overline{\eta}_{i} : i < l\}$ is ordered so that $i < m$ if and only if $l(\eta_{i}) = n$, and similarly for $\{\overline{\nu}_{i} : i < l\}$.  Clearly \(\{\overline{\eta}_{i} : i < l\}\) and \(\{\overline{\nu}_{i} : i < l\}\) are also \(\wedge\)-closed, so we have to check that the two sequences of tuples have the same quantifier type with respect to the relations \(<_{lex}\) and \(\vartriangleleft\).  We'll show this by considering the various cases:

\textbf{Case 1}:  \(i,i' \geq m\).  Then
\begin{eqnarray*}
\overline{\eta}_{i} \vartriangleleft \overline{\eta}_{i'} &\iff& \eta_{i} \vartriangleleft \eta_{i'}
\\ &\iff& \nu_{i} \vartriangleleft \nu_{i'} \\
&\iff& \overline{\nu}_{i} \vartriangleleft \overline{\nu}_{i'} \\
\overline{\eta}_{i} <_{lex} \overline{\eta}_{i} &\iff& \eta_{i} <_{lex} \eta_{i'} \\
 &\iff& \nu_{i} <_{lex} \nu_{i'} \\
 &\iff& \overline{\nu}_{i} <_{lex} \overline{\nu}_{i'}.
\end{eqnarray*}
\textbf{Case 2}:  \(i,i' <m\) and \(j,j' <k\).  Then 
\begin{eqnarray*}
(\overline{\eta}_{i})_{j} \vartriangleleft (\overline{\eta}_{i'})_{j'} &\iff& (\eta_{i} = \eta_{i'}) \wedge (j < j') \\
&\iff & (\nu_{i} = \nu_{i'}) \wedge (j < j') \\
&\iff & (\overline{\nu}_{i})_{j} \vartriangleleft (\overline{\nu})_{j'} \\
(\overline{\eta}_{i})_{j} <_{lex} (\overline{\eta}_{i'})_{j'} &\iff & \eta_{i} <_{lex} \eta_{i'} \vee ( \nu_{i} = \nu_{i'} \wedge j < j') \\
&\iff & \nu_{i} <_{lex} \nu_{i'} \vee (\nu_{i} = \nu_{i'} \wedge j < j') \\
&\iff & (\overline{\nu}_{i})_{j} <_{lex} (\overline{\nu}_{i'})_{j'}.
\end{eqnarray*}

\textbf{Case 3}: \(i < m\), \(i' \geq m\), \(j < k\).  
\begin{eqnarray*}
(\overline{\eta}_{i})_{j} \vartriangleleft \overline{\eta}_{i'}
 &\iff & \eta_{i} \vartriangleleft \eta_{i'} \\
&\iff & \nu_{i} \vartriangleleft \nu_{i'} \\
&\iff & (\overline{\nu}_{i})_{j} \vartriangleleft \overline{\nu}_{i} \\
\overline{\eta}_{i'} \vartriangleleft (\overline{\eta}_{i})_{j} &\iff & \eta_{i'} \vartriangleleft \eta_{i} \\
&\iff & \nu_{i'} \vartriangleleft \nu_{i} \\
&\iff & (\overline{\nu}_{i'})_{j} \vartriangleleft \overline{\nu}_{i} \\
(\overline{\eta}_{i})_{j} <_{lex} \overline{\eta}_{i'} &\iff & \eta_{i} <_{lex} \eta_{i'} \\
&\iff & \nu_{i} <_{lex} \nu_{i'} \\
&\iff & (\overline{\nu}_{i})_{j} <_{lex} \overline{\nu}_{i'} \\
\overline{\eta}_{i'} <_{lex} (\overline{\eta}_{i})_{j} &\iff & \nu_{i'} <_{lex} \nu_{i} \\
&\iff& \nu_{i'} <_{lex} \nu_{i} \\
&\iff & \overline{\nu}_{i'} <_{lex} (\overline{\nu}_{i})_{j}.
\end{eqnarray*}
\end{proof}

\begin{lem}
\begin{enumerate}
 \item Each tuple \(a^{(n)}_{\eta}\) may be enumerated as \((a_{\nu \frown \eta} : \nu \in 2^{n})\)
\item If \((a_{\eta})_{\eta \in 2^{< \kappa}}\) is strongly indiscernible, then for all \(n\), the \(n\)-fold fattening \((a^{(n)}_{\eta})_{\eta \in 2^{< \kappa}}\) is strongly indiscernible over \(C_{n}\)
\end{enumerate}
\end{lem}

\begin{proof}
(1)  This is trivial for \(n = 0\).  Then if true for \(n\), we have 
\[
 a^{(n+1)}_{\eta} = (a^{(n)}_{0 \frown \eta}, a^{(n)}_{1 \frown \eta}) = ((a_{\nu \frown 0 \frown \eta}: \nu \in 2^{n}), (a_{\nu \frown 1 \frown \eta} : \nu \in 2^{n}) ) = (a_{\xi \frown \eta} : \xi \in 2^{n+1}).
\]
(2)  By (1) we have \(a^{(n+1)}_{\eta} = (a_{\mu \frown \eta} : \mu \in 2^{n})\).  Let \(\overline{\mu} = (\mu \in 2^{\leq n})\).  In order to show indiscernibility over \(C_{n}\) have to show that if \(\eta_{0}, \ldots, \eta_{k-1}, \nu_{0}, \ldots, \nu_{k-1} \in 2^{< \omega}\) and 
\[
 \text{qftp}_{L_{0}}(\eta_{0}, \ldots, \eta_{k-1}) = \text{qftp}_{L_{0}}(\nu_{0} ,\ldots, \nu_{k-1})
\]
then \(\text{qftp}_{L_{0}}(\overline{\mu}, (a_{\mu \frown \eta_{0}} : \mu \in {2^{n}} ), \ldots, (a_{\mu \frown \eta_{k-1}} : \mu \in 2^{n}))\) is equal to \(\text{qftp}_{L_{0}}(\overline{\mu}, (a_{\mu \frown \nu_{0}} : \mu \in 2^{n}), \ldots, (a_{\mu \frown \nu_{k-1}} : \mu \in 2^{n}))\).  To this end, we may assume \(\{\eta_{0}, \ldots, \eta_{k-1}\}\) and \(\{\nu_{0}, \ldots, \nu_{k-1}\}\) are meet-closed.  Then \(2^{\leq n} \cup \{\mu \frown \eta_{i} : \mu \in 2^{n}, i < k\}\) and \(2^{\leq n} \cup \{\mu \frown \nu_{i} : \mu \in 2^{n}, i < k\}\) is also meet-closed and we just have to check that the tuples in the above equation have the same time with respect to the language \(L_{t} = \{\vartriangleleft, <_{lex}\}\).  Choose \(\xi_{0}, \xi_{1}\) from the tuple \((\overline{\mu}, (a_{\mu \frown \eta_{0}} : \mu \in {2^{n}} ), \ldots, (a_{\mu \frown \eta_{k-1}} : \mu \in 2^{n}))\) and \(\rho_{0}, \rho_{1}\) from \((\overline{\mu}, (a_{\mu \frown \nu_{0}} : \mu \in 2^{n}), \ldots, (a_{\mu \frown \nu_{k-1}} : \mu \in 2^{n}))\) so that \(\xi_{i}\) sits in the same position in the enumeration of the tuple as \(\rho_{i}\) for \(i = 0,1\).  Now, we must show that \(\xi_{0} <_{lex} \xi_{1}\) if and only if \(\rho_{0} <_{lex} \rho_{1}\) and \(\xi_{0} \unlhd \xi_{1}\) if and only if \(\rho_{0} \unlhd \rho_{1}\).  Choose arbitrary \(\mu_{0}, \mu_{1} \in 2^{\leq n}\), \(\eta_{i}, \eta_{j}\), \(\nu_{i}, \nu_{j}\).  

\textbf{Case 1}:  \(l(\mu_{0}) = l(\mu_{1}) = n\), \(\xi_{0} = \mu_{0} \frown \eta_{i}\), \(\xi_{1} = \mu_{1} \frown \eta_{j}\), and hence \(\rho_{0} = \mu_{0} \frown \nu_{i}\) and \(\rho_{1} = \mu_{1} \frown \nu_{j}\).  
\begin{eqnarray*}
\mu_{0} \frown \eta_{i} \unlhd \mu_{1} \frown \eta_{j} &\iff& \mu_{0} = \mu_{1} \wedge \eta_{i} \unlhd \eta_{j} \\
&\iff& \mu_{0} = \mu_{1} \wedge \nu_{i} \unlhd \nu_{j} \\
&\iff& \mu_{0} \frown \nu_{i} \vartriangleleft \mu_{1} \frown \nu_{j} \\
\mu_{0} \frown \eta_{i} <_{lex} \mu_{1} \frown \eta_{j} &\iff& \mu_{0} <_{lex} \mu_{1} \vee (\mu_{0} = \mu_{1} \wedge \eta_{i} <_{lex} \eta_{j})  \\
&\iff& \mu_{0} <_{lex} \mu_{1} \vee (\mu_{0} = \mu_{1} \wedge \nu_{i} <_{lex} \nu_{i'}) \\
&\iff& \mu_{0} \frown \nu_{i} <_{lex} \mu_{1} \frown \nu_{j}
\end{eqnarray*}

\textbf{Case 2}:  \(\xi_{0} = \mu_{0}\), \(\xi_{1} = \mu_{1}\), \(\rho_{0} = \mu_{0}\), and \(\rho_{1} = \mu_{1}\).  

Clear.

\textbf{Case 3}:  \(l(\mu_{0}) = n\), \(\xi_{0} = \mu_{0} \frown \eta_{i}\), \(\xi_{1} = \mu_{1}\), \(\rho_{0} = \mu_{0} \frown \nu_{i}\), \(\rho_{1} = \mu_{1}\).

It is never the case that \(\mu_{0} \frown \eta_{i} \vartriangleleft \mu_{1}\) or \(\mu_{0} \frown \nu_{i} \vartriangleleft \mu_{1}\) so it suffices to check \(<_{lex}\):
\begin{eqnarray*}
\mu_{0} \frown \eta_{i} <_{lex} \mu_{1} &\iff& \mu_{0} <_{lex} \mu_{1} \\
&\iff& \mu_{0} \frown \nu_{i} <_{lex} \mu_{1}.
\end{eqnarray*}

\textbf{Case 4}:  \(l(\mu_{1}) = n\), \(\xi_{0} = \mu_{0}\), \(\xi_{1} = \mu_{1} \frown \nu_{j}\), \(\rho_{0} = \mu_{0}\), \(\rho_{1} = \mu_{1} \frown \nu_{j}\).  

\begin{eqnarray*}
\mu_{0} \unlhd \mu_{1} \frown \eta_{j} &\iff& \mu_{0} \unlhd \mu_{1} \\
&\iff& \mu_{0} \unlhd \mu_{1} \frown \nu_{j} \\
\mu_{0} \leq_{lex} \mu_{1} \frown \eta_{j} &\iff& \mu_{0} \leq_{lex} \mu_{1} \\
&\iff& \mu_{0} \leq_{lex} \mu_{1} \frown \nu_{j}
\end{eqnarray*}
\end{proof}

\begin{lem}
If \((a_{\eta})_{\eta \in 2^{<\omega}}\) is strongly indiscernible, then for all natural numbers \(k \geq 1\), the \(k\)-fold elongation \((a'_{\eta})_{\eta \in 2^{<\omega}}\) of \((a_{\eta})_{\eta \in 2^{<\omega}}\) is also strongly indiscernible. 
\end{lem}

\begin{proof}
Given \(\eta \in 2^{<\omega}\), with \(l(\eta) = n\), we defined \(\tilde{\eta} \in 2^{<\omega}\) to be the element with length \(k(l(\eta) - 1) + 1\) defined by 
\[
\tilde{\eta}(i) = \left\{
\begin{matrix}
\eta(i/k) & \text{ if } k|i \\
0 & \text{ otherwise}
\end{matrix}
\right.
\]
As the \(k\)\emph{-fold elongation} of \((a_{\eta})_{\eta \in 2^{<\omega}}\) is defined to be the tree \((b_{\eta})_{\eta \in 2^{<\omega}}\) where 
\[
b_{\eta} = (a_{\tilde{\eta}}, a_{\tilde{\eta} \frown 0}, \ldots, a_{\tilde{\eta} \frown 0^{k-1}}).
\]
Write \(\overline{\eta}\) for the tuple \((\tilde{\eta}, \tilde{\eta} \frown 0, \ldots, \tilde{\eta} \frown 0^{k-1})\).  We are reduced to showing that if \(\eta_{0}, \ldots, \eta_{l-1}\), \(\nu_{0}, \ldots, \nu_{l-1}\) are elements of \(2^{<\omega}\) so that 
\[
\text{qftp}_{L_{0}}(\eta_{0}, \ldots, \eta_{l-1}) = \text{qftp}_{L_{0}}(\nu_{0}, \ldots, \nu_{l-1})
\]
then 
\[
\text{qftp}_{L_{0}}(\overline{\eta}_{0}, \ldots, \overline{\eta}_{l-1}) = \text{qftp}_{L_{0}}(\overline{\nu}_{0}, \ldots, \overline{\nu}_{l-1}).  
\]
We may assume that \(\{\eta_{i} : i < l\}\) and \(\{\nu_{i} : i < l\}\) are both \(\wedge\)-closed, from which it follows that \(\{\overline{\eta}_{i} : i < l\}\) and \(\{\overline{\nu}_{i} : i < l\}\) are both \(\wedge\)-closed.  So we must check that \((\overline{\eta}_{i} : i < l)\) and \((\overline{\nu}_{i} : i < l)\) have the same quantifier-free type with respect to the language \(L_{t} = \langle \unlhd, <_{lex} \rangle\).  We note
\begin{eqnarray*}
\tilde{\eta}_{i} \frown 0^{l} \unlhd \tilde{\eta}_{j} \frown 0^{l'} &\iff& \tilde{\eta}_{i} \vartriangleleft \tilde{\eta}_{j} \vee (\tilde{\eta}_{i} = \tilde{\eta}_{j} \wedge l \leq l') \\
&\iff& \eta_{i} \vartriangleleft \eta_{j} \vee (\eta_{i} = \eta_{j} \wedge l \leq l') \\
&\iff & \nu_{i} \vartriangleleft \nu_{j} \vee (\nu_{i} = \nu_{j} \wedge l \leq l') \\
&\iff& \tilde{\nu}_{i} \vartriangleleft \tilde{\nu}_{j} \vee (\tilde{\nu}_{i} = \tilde{\nu}_{j} \wedge l \leq l') \\
&\iff& \tilde{\nu}_{i} \frown 0^{l} \vartriangleleft \tilde{\nu}_{j} \frown 0^{l'} \\
\tilde{\eta}_{i} \frown 0^{l} <_{lex} \tilde{\eta}_{j} \frown 0^{l'} &\iff& \tilde{\eta}_{i} <_{lex} \tilde{\eta}_{j} \vee (\tilde{\eta}_{i} = \tilde{\eta}_{j} \wedge l < l') \\
&\iff& \eta_{i} <_{lex} \eta_{j} \vee (\eta_{i} = \eta_{j} \wedge l < l') \\
&\iff& \nu_{i} <_{lex} \nu_{j} \vee (\nu_{i} = \nu_{j} \wedge l < l') \\
&\iff& \tilde{\nu}_{i} <_{lex} \tilde{\nu}_{j} \vee (\tilde{\nu}_{i} = \tilde{\nu}_{j} \wedge l < l') \\
&\iff& \tilde{\nu}_{i} \frown 0^{l} <_{lex} \tilde{\nu}_{j} \frown 0^{l'}.  
\end{eqnarray*}
\end{proof}

\begin{lem}\label{thehfunction}
Suppose $(a_{\eta})_{\eta \in 2^{<\omega}}$ is a strongly indiscernible tree over $C$.
\begin{enumerate}
\item Define a function $h:2^{<\omega} \to 2^{<\omega}$ by $h(\emptyset) = \emptyset$ and  $h(\eta) = h(\nu) \frown 0 \frown \langle i \rangle$ whenever $\eta = \nu \frown \langle i \rangle$.  Then $(a_{h(\eta)})_{\eta \in 2^{<\omega}}$ is strongly indiscernible over $C$.  
\item For each \(n\), define a map \(h_{n}: 2^{<\omega} \to 2^{<\omega}\) by
\[
h_{n}(\eta) =
\left\{
\begin{matrix}
h(\eta) & \text{ if } l(\eta) \leq n \\
h(\nu) \frown \xi & \text{ if } \eta = \nu \frown \xi, l(\nu) = n.  
\end{matrix}
\right.
\]
Then $(a_{h_{n}(\eta)})_{\eta \in 2^{<\omega}}$ is strongly indiscernible over $C$.
\end{enumerate}
\end{lem}

\begin{proof}
(1)  At the outset, we note that $\eta \unlhd \nu \iff h(\eta) \unlhd h(\nu)$ and $\eta <_{lex} \nu \iff  h(\eta) <_{lex} h(\nu)$.  The only difficulty arises from $\wedge$ which is not preserved by $h$, because if $\eta \perp \nu$ and $\eta \wedge \nu = \xi$ then $h(\eta) \wedge h(\nu) = h(\xi) \frown 0$.  

It suffices to show that if $\overline{\eta}, \overline{\nu}$ are finite tuples from $2^{<\omega}$ with $\text{qftp}_{L_{0}}(\overline{\eta}) = \text{qftp}_{L_{0}}(\overline{\nu})$ then $\text{qftp}_{L_{0}}(h(\overline{\eta})) = \text{qftp}_{L_{0}}(h(\overline{\nu}))$.  Given such $\overline{\eta}, \overline{\nu}$, it is clear that if $\text{qftp}_{L_{0}}(h(\overline{\eta})) \neq \text{qftp}_{L_{0}}(h(\overline{\nu}))$ then $\text{qftp}_{L_{0}}(h(\overline{\eta}')) \neq \text{qftp}_{L_{0}}(h(\overline{\nu}'))$ where $\overline{\eta}'$ and $\overline{\nu}'$ are the $\wedge$-closures of $\overline{\eta}$ and $\overline{\nu}$ respectively.  So we may assume $\overline{\eta}$ and $\overline{\nu}$ are $\wedge$-closed.  We may assume that the tuple $\overline{\eta} = \langle \eta_{i} : i < k \rangle$ is enumerated so that for some $l \leq k$, if $i < l$, then there are $\eta_{j} \perp \eta_{j'}$ so that $\eta_{j} \wedge \eta_{j'} = \eta_{i}$.  It follows that the $\wedge$-closure of $h(\overline{\eta})$ may be enumerated as $\langle h(\eta_{i}) : i < k \rangle \frown \langle h(\eta_{i}) \frown 0 : i < l \rangle$, and, likewise, the $\wedge$-closure of $h(\overline{\nu})$ can be enumerated as $\langle h(\nu_{i}) : i < k \rangle \frown \langle h(\nu_{i}) \frown 0 : i < l \rangle$.  Now we note that, by definition of $h$, if $i,j < k$
\begin{eqnarray*}
h(\eta_{i}) \vartriangleleft h(\eta_{j}) \frown 0 &\iff& h(\eta_{i}) \frown 0 \vartriangleleft h(\eta_{j}) \\ 
&\iff& h(\eta_{i}) \frown 0 \vartriangleleft h(\eta_{j}) \frown 0 \\
& \iff& h(\eta_{i}) \vartriangleleft h(\eta_{j}) \\
h(\eta_{i}) <_{lex} h(\eta_{j}) \frown 0 &\iff& h(\eta_{i}) \frown 0 <_{lex} h(\eta_{j}) \\ 
&\iff& h(\eta_{i}) \frown 0 <_{lex} h(\eta_{j}) \frown 0 \\
& \iff& h(\eta_{i}) <_{lex} h(\eta_{j}) 
\end{eqnarray*}
And similarly for $\nu_{i}, \nu_{j}$.  As $h$ respects $\vartriangleleft$ and $<_{lex}$, and $\text{qftp}_{L_{0}}(\overline{\eta}) = \text{qftp}_{L_{0}}(\overline{\nu})$, it follows that $\text{qftp}_{L_{0}}(h(\overline{\eta})) = \text{qftp}_{L_{0}}(h(\overline{\nu}))$.  

(2) is entirely similar.  
\end{proof}

\bibliographystyle{plain}
\bibliography{treeproperties}

\end{document}